\documentclass[letterpaper,10pt,reqno]{amsart}
\usepackage{indentfirst} 
\usepackage{amssymb}
\usepackage{mathtools} 
\usepackage{mathabx} 
\usepackage[bb=fourier]{mathalfa}  
\usepackage{amsthm}
\usepackage{thmtools}
\usepackage{enumitem} 
\usepackage[backref=page,colorlinks=true]{hyperref} 
\usepackage{ifthen}
\usepackage{tikz}
\usetikzlibrary{decorations.pathreplacing}
\usepackage{tikz-cd}

\renewcommand*{\backref}[1]{}
\renewcommand*{\backrefalt}[4]{\tiny
  \ifcase #1 (\textbf{NOT CITED.})%
  \or    (Cited on page~#2.)%
  \else   (Cited on pages~#2.)%
  \fi}

\makeatletter

\def\MRbibitem{\@ifnextchar[\my@lbibitem\my@bibitem}

\def\mybiblabel#1#2{\@biblabel{{\hyperref{http://www.ams.org/mathscinet-getitem?mr=#1}{}{}{#2}}}}

\def\myhyperanchor#1{\Hy@raisedlink{\hyper@anchorstart{cite.#1}\hyper@anchorend}}

\def\my@lbibitem[#1]#2#3#4\par{%
  \item[\mybiblabel{#2}{#1}\myhyperanchor{#3}\hfill]#4%
  \@ifundefined{ifbackrefparscan}{}{\BR@backref{#3}}%
  \if@filesw{\let\protect\noexpand\immediate
    \write\@auxout{\string\bibcite{#3}{#1}}}\fi\ignorespaces%
}

\def\my@bibitem#1#2#3\par{%
  \refstepcounter\@listctr
  \item[\mybiblabel{#1}{\the\value\@listctr}\myhyperanchor{#2}\hfill]#3%
  \@ifundefined{ifbackrefparscan}{}{\BR@backref{#2}}%
  \if@filesw\immediate\write\@auxout
    {\string\bibcite{#2}{\the\value\@listctr}}\fi\ignorespaces%
}

\makeatother

\DeclareFontFamily{U} {MnSymbolA}{}
\DeclareFontShape{U}{MnSymbolA}{m}{n}{
   <-6> MnSymbolA5
   <6-7> MnSymbolA6
   <7-8> MnSymbolA7
   <8-9> MnSymbolA8
   <9-10> MnSymbolA9
   <10-12> MnSymbolA10
   <12-> MnSymbolA12}{}
\DeclareFontShape{U}{MnSymbolA}{b}{n}{
   <-6> MnSymbolA-Bold5
   <6-7> MnSymbolA-Bold6
   <7-8> MnSymbolA-Bold7
   <8-9> MnSymbolA-Bold8
   <9-10> MnSymbolA-Bold9
   <10-12> MnSymbolA-Bold10
   <12-> MnSymbolA-Bold12}{}
\DeclareSymbolFont{MnSyA} {U} {MnSymbolA}{m}{n}
 \DeclareFontFamily{U} {MnSymbolC}{}
\DeclareFontShape{U}{MnSymbolC}{m}{n}{
  <-6> MnSymbolC5
  <6-7> MnSymbolC6
  <7-8> MnSymbolC7
  <8-9> MnSymbolC8
  <9-10> MnSymbolC9
  <10-12> MnSymbolC10
  <12-> MnSymbolC12}{}
\DeclareFontShape{U}{MnSymbolC}{b}{n}{
  <-6> MnSymbolC-Bold5
  <6-7> MnSymbolC-Bold6
  <7-8> MnSymbolC-Bold7
  <8-9> MnSymbolC-Bold8
  <9-10> MnSymbolC-Bold9
  <10-12> MnSymbolC-Bold10
  <12-> MnSymbolC-Bold12}{}
\DeclareSymbolFont{MnSyC} {U} {MnSymbolC}{m}{n}

\DeclareMathSymbol{\top}{\mathord}{MnSyA}{219} 
\DeclareMathSymbol{\plus}{\mathord}{MnSyC}{20} 

\declaretheorem[numberwithin=section]{theorem}
\declaretheorem[sibling=theorem]{lemma}
\declaretheorem[sibling=theorem]{corollary}
\declaretheorem[sibling=theorem]{proposition}
\declaretheorem[sibling=theorem,style=definition]{definition}
\declaretheorem[sibling=theorem,style=remark]{example}
\declaretheorem[sibling=theorem,style=remark]{conjecture}
\declaretheorem[sibling=theorem,style=remark]{remark}
\declaretheorem[name=Acknowledgements, style=remark, numbered=no]{ack}

\hypersetup{bookmarksdepth = 3} 
\numberwithin{equation}{section}     

\setlist[enumerate,1]{label={\upshape(\alph*)},ref=\alph*}
\setlist[enumerate,2]{label={\upshape(\arabic*)},ref=\arabic*}

\newcommand{\R}{\mathbb{R}}

\newcommand{\N}{\mathbb{N}}
\newcommand{\E}{\mathbb{E}}
\renewcommand{\P}{\mathbb{P}}
\newcommand{\cA}{\mathcal{A}}\newcommand{\cB}{\mathcal{B}}

\newcommand{\cG}{\mathcal{G}}

\newcommand{\cO}{\mathcal{O}}
\newcommand{\cP}{\mathcal{P}}
\newcommand{\cT}{\mathcal{T}}
\newcommand{\cX}{\mathcal{X}}
\newcommand{\cY}{\mathcal{Y}}
\DeclareMathOperator{\per}{per}

\DeclareMathOperator{\pmean}{pm} 
\DeclareMathOperator{\smean}{sm} 
\DeclareMathOperator{\gmean}{gm} 
\DeclareMathOperator{\symean}{sym}
\DeclareMathOperator{\mumean}{mum}

\DeclareMathOperator{\diag}{diag}
\newcommand{\st}{\;\mathord{\colon}\;} 
\renewcommand{\smallint}{{\textstyle \int}\,}
\newcommand{\dd}{\,\mathrm{d}}   
\DeclareMathOperator{\esssup}{ess{\,}sup}
\DeclareMathOperator{\essinf}{ess{\,}inf}

\newcommand{\Rnon}{\mathbb{R}_{\plus}}    
\newcommand{\Rpos}{\mathbb{R}_{\plus\plus}} 
\newcommand{\Mat}[2][]{\ifthenelse{\equal{#1}{}}{\R^{{#2}\times{#2}}}{\R^{{#1}\times{#2}}}}
\newcommand{\Man}[2][]{\ifthenelse{\equal{#1}{}}{\Rnon^{{#2}\times{#2}}}{\Rnon^{{#1}\times{#2}}}}
\newcommand{\Map}[2][]{\ifthenelse{\equal{#1}{}}{\Rpos^{{#2}\times{#2}}}{\Rpos^{{#1}\times{#2}}}}
\newcommand{\dmat}[2]{\Box_{#2}{#1}} 

\newcommand{\logbounded}{\mathcal{B}}  
\newcommand{\logint}{\mathcal{G}}  

\newcommand{\arxiv}[1]{Preprint \href{http://arxiv.org/abs/#1}{arXiv:{#1}}}

\renewcommand{\epsilon}{\varepsilon}
\renewcommand{\phi}{\varphi}

\begin{document}

\title{The Scaling Mean and a Law of Large Permanents}
\date{February, 2015.}


\begin{thanks}
{J.B., G.I.\ and M.P.\ were partially supported by the Center of Dynamical Systems and Related Fields c\'odigo ACT1103 and by FONDECYT projects 1140202, 1110040, and 1140988, respectively.}
\end{thanks}

\author[J.~Bochi]{Jairo Bochi}
\author[G.~Iommi]{Godofredo Iommi}
\author[M.~Ponce]{Mario Ponce}

\maketitle

\begin{abstract}
In this paper we study two types of means of the entries of a nonnegative matrix: the \emph{permanental mean}, which is defined using permanents, and the \emph{scaling mean}, which is defined in terms of an optimization problem. We explore relations between these two means, making use of important results by Ergorychev and Falikman (the van~der~Waerden conjecture), Friedland, Sinkhorn, and others. We also define a scaling mean for functions in a much more general context. Our main result is a Law of Large Permanents, a pointwise ergodic theorem for permanental means of dynamically defined matrices that expresses the limit as a functional scaling mean. The concepts introduced in this paper are general enough so to include as particular cases certain classical types of means, as for example symmetric means and Muirhead means. As a corollary, we reobtain a formula of Hal\'asz and Sz\'ekely for the limit of the symmetric means of a stationary random process.
\end{abstract}


\section{Introduction}


\subsection{Matrix means}

Among the myriad notions of \emph{means} that appear in mathematics, the arithmetic and the geometric means are the most important ones. In this paper we introduce two notions of mean of the entries of a matrix.
Both notions are obtained by combination of the arithmetic and the geometric means.
Let us define them now, leaving details and proofs for later.

The first notion is the \emph{permanental mean} of an $n \times n$ matrix $A$ with nonnegative entries, which is defined as the $n$-th root of the arithmetic mean of the products of the diagonals of $A$. The reason for the name is that this mean is related to the permanent of $A$:
\[
\pmean(A) = \left( \frac{\per A}{n!} \right)^{1/n} \, .
\]
We discuss the main properties of the permanental mean in \S~\ref{ss.pmean}.
Some classic types of means are particular cases of the permanental mean: see Section~\ref{s.appl}.

Two matrices $A$ and $B$ of the same size are called \emph{scalings} of one another if their entries are related by $b_{ij} = x_i a_{ij} y_j$ for suitable vectors $x$, $y$ with positive entries $x_i$, $y_j$. If the geometric means of the entries of both vectors are $1$, we say that $B$ is a \emph{normalized scaling} of $A$. Assuming $A$ to have nonnegative entries, we define the \emph{scaling mean} of $A$, denoted $\smean(A)$, as the infimum of the arithmetic means of all normalized scalings of $A$. It turns out that the scaling mean shares many properties with the permanental mean, and it can be characterized in several other ways: see \S\S~\ref{ss.smean}--\ref{ss.sink}.

This paper is devoted to the study of the relations between the scaling mean and the permanental mean in a range of different settings.

If $A$ is a doubly stochastic $n \times n$ matrix, i.e.\ the arithmetic mean of each row and each column is $1/n$, then by the celebrated van~der~Waerden conjecture (which was an open problem for more than 50 years until being proved in 1981), the permanent of $A$ is at least $n!/n^n$ or, in other words, its permanental mean is at least $1/n$.
On the other hand, as a consequence of the AM-GM inequality (see \S~\ref{ss.smean}), the scaling mean of $A$ equals $1/n$, and therefore
\begin{equation}\label{e.PM-SM_ineq}
\pmean(A) \ge \smean(A)
\end{equation}
for doubly stochastic matrices.
It transpires that this inequality holds true for \emph{all} nonnegative square matrices.
This generalization of the van~der~Warden conjecture is actually a consequence of it, together with the fact that ``most'' nonnegative matrices (in particular all positive ones) are scalings of doubly stochastic matrices, a discovery that dates back from the 1960's with the works of Sinkhorn and others. Details are provided in \S\S~\ref{ss.sink}--\ref{ss.comparison}, where we also characterize the cases of equality in \eqref{e.PM-SM_ineq}.
Moreover, and more importantly for the purposes of this paper,
approximate equality holds in \eqref{e.PM-SM_ineq} for certain ``repetitive'' large matrices:
The precise statement is given in \S~\ref{ss.friedland} and is basically a reinterpretation of a result by Friedland, who in 1979 proved an asymptotic weaker version of the van~der~Waerden conjecture. That fact is one of the foundations for the main result of this paper, which we describe next.

\subsection{The general setting for scaling problems}


Scaling problems have been studied in infinite dimensions, for infinite matrices and for functions.
In this paper we introduce a more general and abstract setting that includes the above ones as particular cases and opens a door for other possibilities.
Basically, we consider measurable functions on arbitrary probability spaces,
and the allowed scaling functions are those that are measurable with respect to certain fixed sub-$\sigma$-algebras;
see Section~\ref{s.functional} for details.

We extend the definition of scaling mean to this general setting.
We also prove the existence of ``doubly stochastic'' scalings under certain boundedness assumptions, thus extending previously known facts about the so-called Sinkhorn decompositions for matrices or $DAD$ problems for functions.

\subsection{A Law of Large Permanents}

On the other hand, it is not clear what should be the permanental mean of an \emph{infinite} matrix.
A natural attempt is to consider square truncations and then take the limit of the corresponding permanental means as the size of the square tends to infinity, provided of course that such limit exists. We prove that this is indeed the case for some classes of dynamically defined matrices. Moreover, we identify the limit permanental mean as the scaling mean of the function that controls the distribution of the matrix entries.
The precise result is as follows:



\begingroup
\def\thetheorem{\ref{t.LLP}}
\begin{theorem}   
Let $(X,\mu)$, $(Y,\nu)$ be Lebesgue probability spaces,
let $T \colon X \to X$ and $S \colon Y \to Y$ be ergodic measure preserving transformations,
and let $f\colon X \times Y \to \R$ be a positive measurable function essentially bounded away from zero and infinity.
Then for $\mu \times \nu$-almost every $(x, y) \in X \times Y$, the permanental mean of the matrix
$$
\begin{pmatrix*}[l]
f(x,y)        & f(Tx,y)        & \cdots & f(T^{n-1}x,y)        \\
f(x,Sy)       & f(Tx,Sy)       & \cdots & f(T^{n-1}x,Sy)       \\
\qquad\vdots  & \qquad\vdots   &        & \qquad\vdots         \\
f(x,S^{n-1}y) & f(Tx,S^{n-1}y) & \cdots & f(T^{n-1}x,S^{n-1}y)
\end{pmatrix*}
$$
converges as $n\to\infty$ to the scaling mean of the function $f$.
\end{theorem}
\addtocounter{theorem}{-1}
\endgroup

In this fairly general pointwise ergodic theorem we not only prove the almost everywhere convergence but we also identify the limit, thus completely solving the two main questions that arise when considering ergodic averages.

Several results scattered in the literature are contained in this \emph{Law of Large Permanents}. For example, the aforementioned result of Friedland  is a corollary. The Law is flexible enough so that we can deduce from it variants of Birkhoff's ergodic theorem where the arithmetic means are replaced by other types of means. For example, in \S~\ref{ss.sym} we deduce an ergodic theorem for symmetric means, thus reobtaining in a  more transparent way  a formula due to Hal\'asz and Szek\'ely \cite{HS}.

Let us remark that the literature contains asymptotic results about permanents of random oblong (i.e.\ non-square) matrices, a subject we will not deal with: see \cite{RW} and references therein. There are fewer results for square matrices: we can only cite \cite{tv}.

\subsection{Organization of the paper}

In Section~\ref{s.2_means} we study permanental and scaling means of matrices, proving the properties mentioned above, among several others.
In Section~\ref{s.functional} we introduce an abstract setting for scaling problems, where we define scaling means for functions and also prove an existence theorem for functional scaling.
Section~\ref{s.LLP} is devoted to the proof of the Law of Large Permanents (Theorem~\ref{t.LLP}).
In Section~\ref{s.appl} we present some corollaries concerning symmetric and Muirhead means.
Section~\ref{s.questions} discusses some of the questions arising from our results,
and poses some conjectures.

\medskip

Throughout this paper we use the following notations:
$$
\Rnon \coloneqq [0,\infty) \, , \quad \Rpos \coloneqq (0,\infty) \, .
$$
The set of real (resp., nonnegative, positive) $m \times n$ matrices
is denoted by $\Mat[m]{n}$ (resp., $\Man[m]{n}$, $\Map[m]{n}$).



\section{Permanental and scaling means of matrices}\label{s.2_means}

\subsection{The permanental mean}\label{ss.pmean}

The \emph{permanent} of a square matrix $A =(a_{ij}) \in \Mat{n}$ is defined as
\[
\per A \coloneqq \sum_{\sigma} \prod_{i=1}^n a_{i,\sigma(i)} \, ,
\]
where $\sigma$ runs on the permutations of $\{1, \dots, n\}$.
This function was introduced independently by Cauchy and Binet around 1812.
It has many applications in combinatorics \cite{VLW}, probability \cite{bapat90}, among other areas.
The book \cite{minc} is wholly dedicated to permanents and contains historical information and the most relevant results and applications up to the late 1970's.

As the determinant, the permanent is a symmetric multilinear function of the rows (or columns) of a matrix; it is also invariant under transposition.
Despite the similarities between the definitions of determinant and permanent, there is no permanental analog of the Gaussian elimination algorithm, and indeed the evaluation of the permanent is a computationally much more complex problem: see \cite[pp.~245ff]{br}.

When dealing with nonnegative matrices, the following function is in some senses better behaved than the permanent itself:

\begin{definition}
The \emph{permanental mean} of a square nonnegative matrix $A \in \Man{n}$ is defined as
\[
\pmean(A) \coloneqq \left( \frac{\per A}{n!} \right)^{1/n} \, .
\]
\end{definition}

The permanental mean has the following properties,
which allow us to consider it as a sort of average of the entries of the matrix:
\begin{itemize}
\item \emph{Monotonicity}: the permanental mean is increasing as a function of each of the entries of the matrix.
\item \emph{Reflexivity}: a matrix with constant entries has a permanental mean equal to this constant.
\end{itemize}
These two properties imply the following one:
\begin{itemize}
\item \emph{Internality}: the permanental mean is between the minimum and the maximum of the entries of the matrix.
\end{itemize}
Some additional properties are:
\begin{itemize}
\item \emph{Continuity}.	
\item \emph{Row-wise and column-wise homogeneity}: the permanental mean on $\Man{n}$
is homogeneous of degree $1/n$ as a function on each row and each column,
and in particular it is a homogenous function of degree $1$.
Equivalently, if $A$ is a nonnegative square matrix and $D$ is a diagonal matrix of the same size and with positive main diagonal, then
$$
\pmean(AD) = \pmean(DA) = \gmean(D) \pmean (A) \, ,
$$
where
$\gmean(D)$ denotes the geometric mean of the entries along the main diagonal of $D$.

\item \emph{Row-wise, column-wise, and transpositional symmetry}: the permanental mean is invariant under permutations of rows or columns, and under transposition of the matrix.
Equivalently, if $A$ is a nonnegative square matrix
and $P$ is a permutation matrix of the same size then
$$
\pmean(AP) = \pmean(PA) = \pmean(A^\top) = \pmean(A) \, .
$$
\end{itemize}

\medskip

A square matrix with nonnegative entries is called \emph{doubly stochastic} if the sums of the entries on each row and each column are equal to $1$.
The set of $n \times n$ doubly stochastic matrices is denoted by $\Omega_n$.
By a classical theorem of Garrett Birkhoff, $\Omega_n$ is a convex polytope  whose vertices are the permutation matrices: see e.g.\ \cite{br,mm,minc}.
A great deal of work has been devoted to study the permanent of this kind of matrices, the van~der~Waerden conjecture being probably the most relevant topic.
If $S \in \Omega_n$ then
\begin{equation}\label{e.VdW}
\frac{n!}{n^n} \leq \per S \leq 1 \, .
\end{equation}
The upper bound is trivial, while the lower bound
was conjectured by van~der~Waerden \cite{vdw} in 1926
and proved in 1981 independently by Egorychev~\cite{egor} and Falikman~\cite{fa}.
Moreover, the minimum of the permanent on $\Omega_n$ is attained at the matrix $J_n$ all of whose entries are equal to $1/n$ and the maximum is attained at permutation matrices.
In terms of the permanental mean, theses bounds become:
\begin{equation}\label{e.VdW2}
\frac{1}{n} \le \pmean(S) \le \frac{1}{(n!)^{1/n}} \quad \text{if $S \in \Omega_n$.}
\end{equation}

\subsection{The scaling mean}\label{ss.smean}

In the following definition, we identify each $x \in \R^n$ with a column vector,
and denote by $\gmean(x)$ denotes the geometric mean of its entries.


\begin{definition}\label{def.smean}
The \emph{scaling mean} of a nonnegative matrix $A \in \Man[m]{n}$ is
\begin{equation}\label{e.def_sm}
\smean(A) \coloneqq \frac{1}{mn} \inf_{x,y}   
\frac{x^\top A y}{\gmean(x)\gmean(y)} \, ,
\end{equation}
where $x$ runs on $\Rpos^n$ and $y$ runs on $\Rpos^m$.
\end{definition}

Though the definition above also applies to oblong matrices, in the sequel we will mostly consider square matrices.

As the permanental mean, the scaling mean has the properties
of monotonicity, reflexivity (a consequence of the AM-GM inequality), and internality,
which is why we call it a ``mean''.

The scaling mean also has the properties of
row-/column-wise homogeneity and row-/column-wise symmetry;
these follow from the analogous properties of the geometric mean.
Transpositional symmetry obviously holds.
Other useful properties are given by the following three propositions:

\begin{proposition}\label{p.coco}
The function $\smean \colon \Man[m]{n} \to \Rnon$ is concave and continuous.
\end{proposition}

\begin{proof}
Being defined as the infimum of a family of linear functions,
the scaling mean is a concave and upper semicontinuous function.
Let us prove lower semicontinuity at an arbitrary fixed $A = (a_{ij})\in \Man[m]{n}$.
Let $H$ be the set of nonnegative matrices  that have the same pattern of zeros
as $A$, i.e.,
$$
H \coloneqq \big\{ (b_{ij}) \in \Man[m]{n} \st b_{ij} = 0 \ \Leftrightarrow a_{ij} = 0 \big\} \, .
$$
Since $H$ is convex and relatively open with respect to its affine hull in $\Mat[m]{n}$, it follows from concavity (see \cite[Thrm.~10.1]{rocka}) that the restriction of the function $\smean$ to $H$ is continuous.
Consider a sequence $\big( A_k = (a_{k,ij}) \big)$ in $\Man[m]{n}$ converging to $A$.
Define $B_k = (a_{k,ij})$ by
$$
b_{k,ij} \coloneqq
\begin{cases}
	a_{k,ij} &\text{if }a_{ij}>0, \\
	0 & \text{otherwise.}
\end{cases}
$$
By monotonicity, $\smean(B_k) \le \smean(A_k)$.
Moreover, $B_k \to A$ and $B_k \in H$ for large enough $k$.
Therefore $\liminf \smean(A_k) \ge \lim \smean(B_k) = \smean(A)$,
thus establishing lower semicontinuity.
\end{proof}

Recall that $\Omega_n \subset \Man{n}$ is the set of doubly stochastic matrices.

\begin{proposition}\label{p.smean_ds}
If $A \in \Omega_n$ then $\smean(A) = 1/n$.
\end{proposition}

\begin{proof}
Let $A \in \Omega_n$.
For all $x$, $y \in \Rpos^n$, by the weighted AM-GM~inequality,
$$
\frac{1}{n} x^\top A y
=    \sum_{i,j} \frac{a_{ij}}{n} \, x_i y_j
\ge \prod_{i,j} (x_i y_j)^{a_{ij}/n}
= \gmean(x) \gmean(y).
$$
Equality is attained if $x=y=(1,\dots,1)$, say.
Therefore $\smean(A) = 1/n$.
\end{proof}

\begin{proposition}\label{p.decomposable}
If $A$ is nonnegative square matrix of block triangular form
$A = \left(\begin{smallmatrix} B & R \\ 0 & C\end{smallmatrix}\right)$
where $B$ and $C$ are square matrices,
then $\smean(A)$ does not depend on $R$.
\end{proposition}

\begin{proof}
Suppose $A$, $B$, and $C$ have  sizes $n \times n$, $k \times k$, and $\ell \times \ell$ respectively, so $n = k + \ell$.
For all $u \in \Rpos$, consider the matrix:
$$
D(u) \coloneqq \diag \big(
{\underbrace{u^{1/k}, \dots, u^{1/k}}_k} ,
{\underbrace{u^{-1/\ell}, \dots, u^{-1/\ell}}_\ell}
\big)\, .
$$
Given $x$, $y \in \Rpos^n$, let $x_u \coloneqq D(u) x$ and $y_u \coloneqq D(u^{-1}) y$.
Then
$$
\gmean(x_u) = \gmean(x), \quad \gmean(y_u) = \gmean(y), \quad\text{and}\quad
x_u^\top A y_u =
x^\top
\begin{pmatrix}
B & u^{1/k+1/\ell} R \\ 0 & C
\end{pmatrix} y .
$$
Therefore
$$
\smean(A) \le
\frac{1}{n^2} \inf_u \frac{x_u^\top A y_u}{\gmean(x_u) \gmean(y_u)} =
\frac{1}{n^2} \frac{x^\top A_0 y}{\gmean(x) \gmean(y)} \, ,
$$
where $A_0$ is the matrix obtained from $A$ by replacing $R$ with $0$.
Taking infimum on $x$, $y$,
we obtain $\smean(A) \le \smean(A_0)$.
On the other hand, $\smean(A) \ge \smean(A_0)$ by monotonicity,
thus concluding the proof.
\end{proof}

\begin{remark}
As a consequence of the AM-GM~inequality, we have:
$$
z \in \Rnon^n \  \Rightarrow \ \inf_{x \in \Rpos^n} \frac{x^\top z}{\gmean(x)} = n \gmean(z).
$$
Therefore \eqref{e.def_sm} can be rewritten as:
\begin{equation}\label{e.London}
A \in \Man[m]{n} \  \Rightarrow \ \smean(A) = \frac{1}{m} \inf_{y \in \Rpos^m} \frac{\gmean (Ay)}{\gmean(y)} \, ,
\end{equation}
As immediate consequences of this formula, we obtain the following additional
properties of the scaling mean:
\begin{align}
A \in \Man[m]{n} , \ B \in \Man[n]{r} \  &\Rightarrow \ \smean(AB) \ge n \smean(A) \smean(B), \\
A \in \Man{n} \  &\Rightarrow \ \smean(A) \le \frac{\rho(A)}{n} \, , \label{e.PF_bound}
\end{align}
where $\rho$ denotes spectral radius.
See also Remarks~\ref{r.London} and \ref{r.other_characterizations}.
\end{remark}

\begin{remark}
The expressions after the $\inf$'s in \eqref{e.def_sm} and \eqref{e.London}
are log-log-convex functions with respect to $(x,y)$ and $y$,
respectively,
so the computation of the scaling mean from either formula
is equivalent to a convex minimization problem.
\end{remark}

\begin{remark}
One may ask if the permanental mean also has
the concavity property obtained for the scaling mean in Proposition~\ref{p.coco}.
The answer is no,
because otherwise the minimum of permanental mean on the convex polytope $\Omega_n$ would be attained at the boundary, while we know by the Egorychev--Falikman theorem that this is not the case. 
On the other hand, the permanental mean is indeed concave among positive definite symmetric matrices: see \cite[p.~282]{bhatia}. 
\end{remark}

\subsection{Matrix scaling and Sinkhorn decompositions}\label{ss.sink}

Recall from the introduction that two matrices $A$, $B \in \Mat[m]{n}$
are called \emph{scalings} of one another
if there exist diagonal matrices $E \in \Mat{m}$, $D \in \Mat{n}$ with positive main diagonals such that $A = EBD$.
Matrix scaling has applications in numerical analysis \cite{gvl} and economics \cite{ls}.
We will be especially interested in scaling when one of the matrices is
doubly stochastic,
so we introduce the following:

\begin{definition}
A \emph{Sinkhorn decomposition} of a square matrix $A$ is a factorization of the form
\begin{equation*}
A = DSE \quad \text{where}
\left\{
\begin{array}{l}
	\text{$S$ is doubly stochastic,}\\
	\text{$D$, $E$ are diagonal with positive main diagonals.}
\end{array}
\right.
\end{equation*}
\end{definition}

In 1964, Sinkhorn \cite{si64} proved that any
positive square matrix can be scaled to a doubly stochastic matrix
and moreover the corresponding Sinkhorn decomposition
is unique up to multiplication of $D$ and $E$ by positive factors
$\lambda$ and $\lambda^{-1}$, respectively.
Not all nonnegative square matrices possess Sinkhorn decompositions, however;
an example is $\left( \begin{smallmatrix} 1 & 1 \\ 0 & 1 \end{smallmatrix} \right)$.

\medskip

The following proposition relates the previously defined means
and Sinkhorn decompositions.
Recall that if $D$ is a diagonal matrix with positive entries along the main diagonal
then $\gmean(D)$ denotes the geometric mean of these entries.

\begin{proposition}\label{p.smean_via_sink}
If $A \in \Man{n}$ has a Sinkhorn decomposition $A=DSE$ then
\begin{equation}\label{e.smean_via_sink}
\smean(A) = 
\frac{1}{n} \, \gmean(D) \gmean(E)
= \frac{1}{n} \, \frac{\pmean(A)}{\pmean(S)} \, .
\end{equation}
Moreover, if $x$ and $y$ are the vectors whose entries form the diagonals of $D^{-1}$ and $E^{-1}$, respectively, then the infimum in \eqref{e.def_sm} is attained at $x$, $y$.
\end{proposition}

\begin{proof}
The first part is an immediate consequence of Proposition~\ref{p.smean_ds}
and the row-/column-wise homogeneity of the scaling and the permanental means.
If $x$ and $y$ are as described, then
$$
\frac{1}{n^2} \, \frac{x^\top A y}{\gmean(x)\gmean(y)}  =
\frac{1}{n} \, \gmean(D)\gmean(E) \, ,
$$
so the second part follows from the first one.
\end{proof}

\begin{example}
The Sinkhorn decomposition of a $2 \times 2$ positive matrix is given by
\[
\begingroup
\renewcommand*{\arraystretch}{1.5}
\begin{pmatrix}
a & b \\
c & d
\end{pmatrix} =
\begin{pmatrix}
\frac{\sqrt{ad}+\sqrt{bc}}{\sqrt{cd}} & 0 \\
0 & \frac{\sqrt{ad}+\sqrt{bc}}{\sqrt{ab}}
\end{pmatrix}
\begin{pmatrix}
\frac{\sqrt{ad}}{\sqrt{ad}+\sqrt{bc}} & \frac{\sqrt{bc}}{\sqrt{ad}+\sqrt{bc}} \\
\frac{\sqrt{bc}}{\sqrt{ad}+\sqrt{bc}} & \frac{\sqrt{ad}}{\sqrt{ad}+\sqrt{bc}}
\end{pmatrix}
\begin{pmatrix}
\sqrt{ac} & 0 \\
0 & \sqrt{bd}
\end{pmatrix} \, ,
\endgroup
\]
and therefore the scaling mean is
\begin{equation*} 
\smean \begin{pmatrix}
a & b \\
c & d
\end{pmatrix} = \frac{\sqrt{ad} + \sqrt{bc}}{2} \, .
\end{equation*}
By continuity (Proposition~\ref{p.coco}), this formula also holds for nonnegative matrices.
\end{example}

\begin{remark}\label{r.London}
Relations \eqref{e.London} and \eqref{e.PF_bound}, once
rewritten using the first equality from \eqref{e.smean_via_sink},
were stated explicitly and proved by London~\cite{london}.
\end{remark}

\subsection{Existence of Sinkhorn decompositions}

We need some definitions:

\begin{itemize}
\item Two matrices of the same size have the \emph{same zero pattern}
if their zero entries occupy the same positions.

\item A \emph{diagonal} of a square matrix is a sequence of entries containing exactly one entry from each row and one entry from each column; a diagonal is called \emph{positive} if each of its elements is positive.

\item Two matrices $A$, $B \in \Mat[m]{n}$ are \emph{permutationally equivalent} if
there exist permutation matrices $P \in \Mat{m}$, $Q\in\Mat{n}$ such that $B = PAQ$.

\item A matrix $A\in \Mat{n}$
called \emph{fully indecomposable} if either $n=1$ and $A\neq 0$
or $n \ge 2$ and $A$ is not permutationally equivalent
to a matrix of the form $\left(\begin{smallmatrix} B & R \\ 0 & C\end{smallmatrix}\right)$,
where $B$ and $C$ are square matrices.

\item The \emph{direct sum} of two square matrices $A$, $B$ is the
matrix $\left(\begin{smallmatrix} A & 0 \\ 0 & B\end{smallmatrix}\right)$;
this is an associative (but non commutative) operation.
\end{itemize}

\begin{theorem}[Perfect--Mirsky \cite{pm}]\label{t.pm}
Let $A \in \Man{n}$.
The following assertions are equivalent:
\begin{enumerate}
\item\label{i.pattern}
$A$ has the zero pattern of some doubly stochastic matrix;
\item\label{i.diagonals}
every positive entry of $A$ lies on a positive diagonal;
\item\label{i.FI}
$A$ is permutationally equivalent to a direct sum of fully indecomposable matrices.
\end{enumerate}
Moreover, if $n \ge 2$ then the assertions above are equivalent to:
\begin{enumerate}[resume]
\item\label{i.blocks}
$A$ is \emph{not} permutationally equivalent to a matrix of the form
$\left( \begin{smallmatrix} B & R \\ 0 & C\end{smallmatrix} \right)$,
where the matrices $B$ and $C$ are square and $R$ is nonzero.
\end{enumerate}
\end{theorem}

\begin{definition}\label{def.Pn}
Let $\cP_n$ be the subset of $\Man{n}$ formed by matrices that satisfy any of the equivalent conditions in Theorem~\ref{t.pm}.
\end{definition}

The following theorem provides still another equivalent definition for the set $\cP_n$:

\begin{theorem}[General Sinkhorn decompositions]\label{t.general_Sinkhorn}
A matrix $A \in \Man{n}$ has a Sinkhorn decomposition $A=DSE$
if and only if $A \in \cP_n$,
and in that case the doubly stochastic matrix $S$ is unique.
Moreover, the map $A \mapsto S$ is continuous.
\end{theorem}
The first part of the theorem was proved independently and by means of different arguments
by several authors \cite{bps,sk,mo}.
Continuity was proved by Sinkhorn~\cite{si72} (see also \cite{menon68}).
Since these early papers, many other proofs, generalizations, and numerical studies
have appeared in the literature.

Some of the proofs of existence of Sinkhorn decompositions are closely related to the scaling mean (and indeed motivated our Definition~\ref{def.smean}), so let us sketch them.

In order to prove that matrices in $\cP_n$
have Sinkhorn decompositions, it is sufficient to consider
fully indecomposable matrices~$A$,
because then the general case follows by characterization~(\ref{i.FI}) in Theorem~\ref{t.pm}.
In that case, Marshall and Olkin~\cite{mo}
have shown that the function after the $\inf$ in \eqref{e.def_sm}
diverges to infinity as $(x,y)$ approaches the boundary of the cone $\Rpos^n \times \Rpos^n$.
In particular, the infimum is attained at some $(x,y)$ in the interior of the cone.
Then a Lagrange multipliers calculation shows that:
$$
\left.
\begin{array}{l}
\lambda \coloneqq  x^\top A y  / n  \\
D \coloneqq \lambda \diag(x_1^{-1},\dots,x_n^{-1}) \\
E \coloneqq \diag(y_1^{-1},\dots,y_n^{-1})
\end{array}
\right\}
\ \Rightarrow
S \coloneqq D^{-1} A E^{-1}
\text{ is doubly stochastic,}
$$
and so $A=DSE$ is the sought after Sinkhorn decomposition.
Notice that this strategy is basically a converse of Proposition~\ref{p.smean_via_sink}.
Djokovi\'{c}~\cite{dj} and London~\cite{london} independently
discovered an analogous proof based instead on the minimization problem \eqref{e.London}.
For a related discussion, see \cite[\S~2]{bapat86}.

\smallskip

Let us introduce a technical device that will be useful later.
Given $A \in \Man{n}$, let $\Pi(A)$
denote the matrix obtained by
keeping all entries that lie on positive diagonals,
and setting the remaining entries to zero.
The map $\Pi$ is a projection, i.e.\ $\Pi \circ \Pi = \Pi$,
and although it is discontinuous it has useful properties:

\begin{proposition}\label{p.Pi}
If $A \in \Man{n}$ then:
\begin{enumerate}
\item\label{i.image_Pi}
$\Pi(A) \in \cP_n \cup \{0\}$;
\item\label{i.kernel_Pi}
$\Pi(A) = 0$ iff $\per A = 0$;
\item\label{i.pmean_Pi}
$\pmean \Pi(A) = \pmean(A)$;
\item\label{i.smean_Pi}
$\smean \Pi(A) = \smean(A)$.
\end{enumerate}
\end{proposition}

\begin{proof}
Property~(\ref{i.image_Pi}) follows from characterization~(\ref{i.diagonals}) in Theorem~\ref{t.pm},
properties~(\ref{i.kernel_Pi}) and (\ref{i.pmean_Pi})
follow from the definition of the permanent,
and property~(\ref{i.smean_Pi}) follows from
from characterization~(\ref{i.blocks}) in Theorem~\ref{t.pm},
the row-/column-wise symmetry of the scaling mean, and Proposition~\ref{p.decomposable}.
\end{proof}

%
%
%
%
%

%
%

\subsection{Comparison between the two means}\label{ss.comparison}

We have the following inequalities between the two means we have defined:

\begin{theorem}[Generalized van~der~Waerden bounds]\label{t.generalized_VdW}
For all $A \in \Man{n}$,
\begin{equation}\label{e.comparable_means}
\smean(A) \le \pmean(A) \le n (n!)^{-1/n} \, \smean(A) \, .
\end{equation}
Equality holds in the first inequality iff $A$ has permanent $0$ or rank $1$.
Equality holds in the second inequality iff $A$ has permanent $0$ or has the zero pattern of a permutation matrix.
\end{theorem}

The sequence $\big( n (n!)^{-1/n} \big)$ is increasing, and by Stirling's formula it converges to~$e$; in particular:
\begin{equation}\label{e.comparable_means_2}
\smean(A) \le \pmean(A) <  e \, \smean(A) \quad \text{for all nonnegative square matrices } A \, .
\end{equation}

Concerning the cases of equality in Theorem~\ref{t.generalized_VdW}, we recall that by the Frobenius--K\"onig theorem \cite{mm,minc}, a matrix in $\Man{n}$ has  zero permanent if and only if it contains a zero submatrix of size $r \times s$ with $r+s=n+1$.

Notice that if $A \in \Omega_n$ then, by Proposition~\ref{p.smean_ds}, inequalities~\eqref{e.comparable_means} are equivalent to \eqref{e.VdW}; in particular, the first inequality in~\eqref{e.comparable_means} is a generalization of the van~der~Waerden bound to general nonnegative matrices.
Actually, \eqref{e.comparable_means} is a corollary of \eqref{e.VdW}, as we proceed to show.

\begin{proof}[Proof of Theorem~\ref{t.generalized_VdW}]
First consider $A \in \cP_n$.
By Theorem~\ref{t.general_Sinkhorn}, $A$ has a Sinkhorn decomposition $DSE$.
The matrix $S$ obeys the bounds \eqref{e.VdW}
or equivalently \eqref{e.VdW2}, and so using Proposition~\ref{p.smean_via_sink}
we obtain \eqref{e.comparable_means}.
The first inequality in \eqref{e.comparable_means} is an equality
iff $A = D^{-1} J_n E^{-1}$, that is, iff $A$ is a positive matrix of rank $1$.
The second inequality is an equality iff $A = D^{-1} P E^{-1}$ for some permutation matrix $P$,
that is, iff $A$ has the zero pattern of a permutation matrix.

Since $\cP_n$ is dense in $\Man{n}$,
by continuity (Proposition~\ref{p.coco})
we conclude that inequalities \eqref{e.comparable_means} hold for every $A$.

The stated conditions for equality are clearly sufficient, so let us check necessity.
Consider $A$ such that $\smean(A) = \pmean(A)$.
By Proposition~\ref{p.Pi}, $\smean \Pi(A) = \pmean \Pi(A)$
and moreover there are two possibilities:
either $\Pi(A) = 0$ or $\Pi(A) \in \cP_n$.
In the first case, $A$ has permanent $0$.
In the second case, as we have seen above, $\Pi(A)$ is a positive matrix of rank~$1$;
in particular $A = \Pi(A)$ has rank $1$.
This proves the characterization of the first equality.
The second one is dealt with analogously.
\end{proof}

Let us review a few other results of the literature from our perspective.
The paper \cite{lsw} basically shows that the scaling mean
can be computed efficiently, and therefore by \eqref{e.comparable_means_2}
the permanental mean can be efficiently computed up to a factor of $e$,
or equivalently, the permanent of a nonnegative $n \times n$ matrix
can be efficiently computed up to a factor of $e^n$.
Using instead the Bethe approximation for the permanent,
it is possible to improve this factor to $2^n$: see \cite{gs}.
We will see next further advantages of the scaling mean
for the approximation of permanents.

\subsection{The scaling mean as a limit of permanental means}\label{ss.friedland}

Recall that the \emph{Kronecker product} of two arbitrary matrices is defined as:
$$
A \otimes B \coloneqq
\begin{pmatrix}
a_{11}B & \dots & a_{1n}B \\
\vdots  &       &  \vdots \\
a_{m1}B & \dots & a_{mn}B
\end{pmatrix}
\in \Mat[mr]{ns} \text{ if }
A = (a_{ij}) \in \Mat[m]{n}, \  B \in \Mat[r]{s} \, .
$$
It satisfies the following \emph{mixed-product property} (see \cite[p.~9]{mm}):  
\begin{equation}\label{e.mixed}
(A \otimes B)(C \otimes D) = (AC) \otimes (BD) \, .
\end{equation}
Notice that the map $\Pi$ considered in Proposition~\ref{p.Pi}
has the following additional property:
\begin{equation}\label{e.Kronecker_Pi}
\Pi(A \otimes B) = \Pi(A) \otimes B
\quad \text{if $B$ is a \emph{positive} square matrix.}
\end{equation}

The scaling mean is well behaved with respect to the Kronecker product:
\begin{proposition}\label{p.Kronecker_smean}
If $A$, $B$ are nonnegative square matrices then
$$
\smean(A \otimes B) = \smean(A) \smean(B).
$$
\end{proposition}

\begin{proof}
It is possible to prove the proposition directly from the definition of scaling mean, but let us give an alternative argument.
By continuity, it is sufficient to consider matrices
$A \in \Man{n}$, $B \in \Man{m}$ that have Sinkhorn decompositions,
say $A = DSE$, $B = D' S' E'$.
Then $A \otimes B$ also has a Sinkhorn decomposition, namely,
$$
A \otimes B = (D \otimes D')(S \otimes S')(E \otimes E') \, ,
$$
so it follows from Proposition~\ref{p.smean_via_sink}
that
\begin{align*}
\smean(A \otimes B)
&= \frac{\gmean(D \otimes D') \gmean(E \otimes E')}{nm} \\
&= \frac{\gmean(D) \gmean(E)}{n} \cdot \frac{\gmean(D') \gmean(E')}{m} \\
&= \smean(A)\smean(B) \, . \qedhere
\end{align*}
\end{proof}

The property expressed by Proposition~\ref{p.Kronecker_smean}
fails for permanental means: consider identity matrices, for example.
On the other hand, Brualdi \cite{brualdi} formulated a conjecture
that can be restated in terms of permanental means as follows:
If $A$, $B$ are nonnegative square matrices then
\begin{equation}\label{e.Brualdi}
\pmean (A \otimes B) \leq \pmean(A) \pmean (B) \, .
\end{equation}
Based on experimental evidence, we also conjecture that equality holds
iff $A$ or $B$ have permanent $0$ or both $A$ and $B$ have rank $1$.
Unfortunately Brualdi's conjecture, formulated around 50~years ago,  apparently has not received much attention.

\medskip

We now describe an important relation between the permanental and the scaling means.
Prior to the proof of the van~der~Waerden conjecture, Friedland \cite{fr}
proved that the permanent of each $S \in \Omega_n$ is at least $e^{-n}$,
which by the Stirling formula
differs from the van~der~Waerden lower bound by a subexponential factor.
The crux of his proof is to obtain the following limit formula:
\begin{equation}\label{e.friedland_lim}
\lim_{m \to \infty} (\per(S \otimes J_m))^{1/m} = e^{-n}
\quad \text{if $S \in \Omega_n$,}
\end{equation}
where $J_m$ is the $m \times m$ matrix all of whose entries are $1/m$.
We call this fact the \emph{Friedland limit};
it appears in \cite{fr} as formula~(1.6)
and, as explained in \S~2 of that paper,
follows relatively easy from the then unproved van~der~Waerden bound.
Basically as corollary of Friedland's limit,
we will obtain the following:

\begin{theorem}[Generalized Friedland limit] \label{t.friedland}
For any nonnegative square matrix $A$ we have
\begin{equation}\label{e.generalized_friedland}
\smean(A) = \lim_{m \to \infty} \pmean(A \otimes U_m)  \, ,
\end{equation}
where $U_m$ is the $m \times m$ matrix all of whose entries are $1$.
\end{theorem}

\begin{remark}
The generalized van~der~Waerden bound
(i.e., the first inequality in Theorem~\ref{t.generalized_VdW})
becomes a corollary of Theorem~\ref{t.friedland}
if Brualdi conjecture \eqref{e.Brualdi} is assumed.
\end{remark}

\begin{proof}[Proof of Theorem~\ref{t.friedland}]
Consider $S \in \Omega_n$.
By Stirling formula, $(m!)^{1/m} = e^{-1} \, m \, \theta_m${\,}, where $\theta_m \to 1$, so
$$
\pmean (S \otimes U_m) = m \pmean (S \otimes J_m)
= m \left(\frac{\per(S \otimes J_m)}{(nm)!}\right)^{1/nm}
= \frac{\left(\per(S \otimes J_m)\right)^{1/nm}}{e^{-1} \, n \, \theta_{nm}} \, ,
$$
which by Friedland's limit \eqref{e.friedland_lim} converges to $1/n$.
Since $\smean(S) = 1/n$, 
we conclude that \eqref{e.generalized_friedland} holds for
doubly stochastic matrices.
Next consider a matrix $A$ in the set $\cP_n$ (recall Definition~\ref{def.Pn}), and let $DSE$ be its Sinkhorn decomposition.
Using the mixed-product property~\eqref{e.mixed}
we factorize $A \otimes U_m$ as $(D \otimes I_m)(S \otimes U_m)(E \otimes I_m)$,
where the matrices $D \otimes I_m$ and $E \otimes I_m$ are diagonal.
Therefore
$$
\pmean (A\otimes U_m) =
{\underbrace{\gmean(D\otimes I_m)}_{=\gmean(D)}} \,
{\underbrace{\pmean(S \otimes U_m)}_{\to 1/m}}   \,
{\underbrace{\gmean(E\otimes I_m)}_{=\gmean(E)}}
\to \smean(A) \, ,
$$
establishing \eqref{e.generalized_friedland} in this case.
Finally, consider a general $A \in \Man{n}$.
Using Proposition~\ref{p.Pi} and property~\eqref{e.Kronecker_Pi}
we obtain
\[
\pmean(A \otimes U_m) = \pmean(\Pi(A\otimes U_m)) = \pmean(\Pi(A)\otimes U_m)
\to \smean(\Pi(A)) = \smean(A). \qedhere
\]
\end{proof}

As mentioned in the end of \S~\ref{ss.comparison}, the scaling mean is computationally easier to compute than the permanental mean, and in view of the bounds of Theorem~\ref{t.generalized_VdW}, the former can be used to approximate the latter up to a multiplicative error of at most $e$. Theorem~\ref{t.friedland} indicates that for large matrices of the form $A \otimes U_m$, this factor is actually close to $1$.

As we will see later, some classes of random matrices are approximately permutationally equivalent to matrices of the form $A \otimes U_m$: this is the idea behind the Law of Large Permanents.
While in this paper we are not concerned with numerical studies, it may be interesting to investigate further the effectiveness of the scaling mean for the computation of permanents of reasonably well-behaved matrices.

\begin{remark}\label{r.other_characterizations}
In addition to \eqref{e.def_sm}, \eqref{e.London}, \eqref{e.smean_via_sink}, and \eqref{e.generalized_friedland},
other characterizations of the scaling mean are
\cite[Thrm.~5.3]{fls}, \cite[formula~(2)]{gs} and the following one:
\begin{equation}\label{e.smean_via_PF}
\smean(A) = \frac{1}{n} \inf_{\Delta} \frac{\rho(\Delta A)}{\gmean(\Delta)} \, ,
\end{equation}
where $\Delta$ runs over the diagonal matrices with positive main diagonal.
We only sketch the proof:
The $\le$ inequality follows from \eqref{e.PF_bound} and row-wise homogeneity of the scaling mean.
In the case that $A \in \cP_n$, i.e.\ $A$ has a Sinkhorn decomposition $DSE$,
the infimum on the RHS is attained at $\Delta = E^{-1} D^{-1}$.
Recall that the LHS of \eqref{e.smean_via_PF} depends continuously on $A$,
while the RHS is upper semicontinuous and monotonically increasing with respect to matrix entries.
(The latter fact follows from monotonicity of $\rho$, itself a consequence of Gelfand spectral radius formula \cite[p.~204]{bhatia}.)
This permits us to extend the equality from the dense set $\cP_n$ to the whole $\Man{n}$.
\end{remark}

\section{Scaling mean and Sinkhorn decomposition of functions} \label{s.functional}


In this section we extend the notions of scaling mean and Sinkhorn decomposition from matrices to arbitrary nonnegative measurable functions defined on probability spaces with respect to a pair of sub-$\sigma$-algebras. The allowable scalings are those that are measurable with respect to one of these sub-$\sigma$-algebras, which thus supersede the row/column structure.

Moreover we prove the existence of this functional Sinkhorn decomposition under a boundedness assumption (Theorem~\ref{t.sink_for_bounded}).

Finally, in \S~\ref{ss.XY} we analyze the particular setting of direct products, which is the one that we will use later for our Law of Large Permanents (Section~\ref{s.LLP}) and for its applications (Section~\ref{s.appl}).
Let us remark that though this particular setting is sufficient for the main results of this paper, the general setting where scaling are controlled by sub-$\sigma$-algebras is in our opinion more transparent and, as discussed in Section~\ref{s.questions}, should be indispensable for stronger Laws.

\subsection{The functional scaling mean}\label{ss.func_smean}

Let us fix a probability space $(\Omega, \cA, \P)$.
Let $\logint(\P)$ denote the set of positive measurable functions $\phi \colon \Omega \to \Rpos$
such that $\log \phi \in L^1(\P)$.
The \emph{geometric mean} of a function $f \in \logint(\P)$ is defined as:
\begin{equation}\label{e.func_gmean}
\gmean(\phi) \coloneqq \exp \left( \smallint \log \phi \dd \P \right) \, .
\end{equation}
The AM-GM inequality says that $\gmean (\phi) \le \int \phi \dd\P$.

Let us also fix a pair of sub-$\sigma$-algebras $\cA_1$, $\cA_2\subset \cA$.
For each $i\in \{1, 2\}$, we define two sets of functions $\cG_i \supset \cB_i$ as follows:
\begin{align}
\cG_i &\coloneqq \big\{ \phi \colon \Omega \to \Rpos \st
\text{$\phi$ is $\cA_i$-measurable and $\log \phi \in L^1(\P)$} \big\} \, , \label{e.def_Gi}
\\
\cB_i &\coloneqq \big\{ \phi \colon \Omega \to \Rpos \st
\text{$\phi$ is $\cA_i$-measurable and $\log \phi \in L^\infty(\P)$} \big\} \, . \label{e.def_Bi}
\end{align}
Our scaling functions will always take values in these spaces.

\begin{definition} \label{def.funct.smean}
Let {$f\colon \Omega \to \Rnon$} be a nonnegative measurable function.
The \emph{scaling mean} of $f$ with respect to the sub-$\sigma$-algebras $\cA_1$, $\cA_2$ is defined as:
\begin{equation}\label{e.def_func_sm}
\smean_{\cA_1,\cA_2}(f) \coloneqq  \inf_{\substack{g_1 \in \cG_1 \\ g_2 \in \cG_2}}
\frac{1}{\gmean(g_1) \gmean(g_2)} \int g_1 f g_2 \dd \P  \, .
\end{equation}
When no confusion is likely to arise, we write $\smean(f) = \smean_{\cA_1,\cA_2}(f)$.
\end{definition}

Concrete examples will be presented later in \S~\ref{ss.XY}.

We list some basic properties of the scaling mean:
\begin{itemize}
\item \emph{Monotonicity:} If $f \le g$ almost everywhere then $\smean(f) \le \smean(g)$.
\item \emph{Reflexivity:} If $f$ equals a constant $c$ almost everywhere then $\smean(f) = c$; this is a consequence of the  AM-GM inequality.
\item \emph{Homogeneity:} If $g_1 \in \cG_1$ and $g_2  \in \cG_2$ then
\begin{equation}\label{e.func_smean_homog}
\smean (g_1 f g_2 ) = \gmean(g_1) \smean (f) \gmean(g_2 ) \, .
\end{equation}
\end{itemize}

The following proposition says that in order to evaluate the infimum in formula \eqref{e.def_func_sm} it is sufficient to consider $g_i$ in the smaller space $\cB_i$ defined by \eqref{e.def_Bi}:

\begin{proposition}\label{p.sm_bounded}
For any measurable {$f\colon \Omega \to \Rnon$} we have:
$$
\smean(f) \coloneqq  \inf_{\substack{g_1 \in \cB_1 \\  g_2 \in \cB_2}}
\frac{1}{\gmean(g_1) \gmean(g_2)} \int g_1 f g_2 \dd \P  \, .
$$
\end{proposition}

\begin{proof}
Fix $f \ge 0$ and consider $g_1 \in \cG_1$, $g_2 \in \cG_2$ such that $\int g_1 f g_2 \dd \P < \infty$.
Define two sequences $(g_{1,k})$, $(g_{2,k})$ respectively in $\cB_1$, $\cB_2$ as follows:
$$
g_{i,k}(\omega) \coloneqq
\begin{cases}
	g_i(\omega) &\quad\text{if $|\log g_1(\omega)| \le k$,} \\
	1 &\quad\text{otherwise.}
\end{cases}
$$
Applying the dominated convergence theorem three times, we conclude that:
$$
\frac{1}{\gmean(g_1) \gmean(g_2)} \int g_1 f g_2 \dd \P
= \lim_{k\to\infty} \frac{1}{\gmean(g_{1,k}) \gmean(g_{2,k})} \int g_{1,k} f g_{2,k} \dd \P  \, .
$$
The proposition follows.
\end{proof}

\subsection{Conditional expectations and doubly stochastic functions} \label{ss.condex_ds}

We begin by recalling some basic facts about conditional expectations, referring the reader to \cite{boga} for more details. 

Let $(\Omega, \cA , \P)$ be a probability space. If $f \in L^1(\P)$ and $\cA_1$ is a sub-$\sigma$-algebra of $\cA$, let $\E(f|\cA_1)$ denote the \emph{conditional expectation} of $f$ with respect to $\cA_1$, i.e.\ the a.e.-unique $\cA_1$-measurable function such that
$$
\int g f \dd\P =  \int g \, \E(f|\cA_1) \dd\P  \quad \text{for every bounded $\cA_1$-measurable function $g$.}
$$
For example, if $\{B_1, \dots B_k\}$ is a partition of $\Omega$ into finitely many sets of positive probability and $\cA_1$ is the $\sigma$-algebra generated by this partition, then $\E(f|\cA_1)$ is the simple function that takes the value $\frac{1}{\P(B_i)} \int_{B_i} f \dd\P$ on the set $B_i$.

Note that if $f \in L^1(\P)$ then,
as an immediate consequence of the definition of conditional expectation,
\begin{equation}\label{e.condex_product}
\E (g f | \cA_1)  =  g \, \E(f | \cA_1) \quad \text{for every bounded $\cA_1$-measurable function $g$.}
\end{equation}


\medskip

Again fix a probability space $(\Omega, \cA, \P)$
and a pair of sub-$\sigma$-algebras $\cA_1$, $\cA_2 \subset \cA$.
Let us say that a integrable nonnegative function $f\colon \Omega \to \Rnon$ is \emph{doubly stochastic} if
\begin{equation}\label{e.def_func_dsO}
\E(f|\cA_1) = \E(f|\cA_2) = 1 \quad \P \textrm{-a.e.}
\end{equation}
Notice that a doubly stochastic function has arithmetic mean $\smallint f \dd\P = 1$.
Analogously for the scaling mean:

\begin{proposition}\label{p.func_smean_ds}
If $f$ is doubly stochastic then $\smean(f) = 1$.
\end{proposition}

\begin{proof}
If $f$ is doubly stochastic then the measure $\dd\nu \coloneqq f \dd\P$ is a probability.
Therefore, for any $g_1 \in \cB_1$ and $g_2 \in \cB_2$,
using the AM-GM inequality  and property~\eqref{e.condex_product} we obtain:
\begin{align*}
\smallint g_1 f g_2 \dd\P
&=   \smallint g_1g_2 \dd\nu\\
&\ge \exp\left(\smallint \log (g_1g_2) \dd\nu \right) \\
&=    \exp\left ( \smallint f\log g_1 \dd\P \right)  \exp\left ( \smallint f\log g_2 \dd\P \right)\\
&=    \exp \left(\smallint \E(f\log g_1 | \cA_1) \dd\P\right)  \exp\left(\smallint \E(f\log g_2 | \cA_2) \dd\P \right)\\
&=    \exp \left(\smallint (\log g_1)\E(f | \cA_1)\dd\P\right)  \exp\left(\smallint (\log g_2)\E(f | \cA_2) \dd\P \right)\\
&=    \gmean(g_1)\gmean(g_2).
\end{align*}
So it follows from Proposition~\ref{p.sm_bounded} that $\smean(f)\geq 1$.
Considering $g_1 = g_2$ identically equal to $1$ we conclude that $\smean(f)=1$.
\end{proof}

\subsection{Functional Sinkhorn decompositions}

Fix a probability space $(\Omega, \cA, \P)$ and a pair of sub-$\sigma$-algebras $\cA_1$, $\cA_2 \subset \cA$.

\begin{definition}\label{def.func_sink}
A \emph{Sinkhorn decomposition} of a function  $f\colon \Omega \to \Rnon$
is a factorization of the form
$$
f(\omega) = \phi(\omega) g(\omega) \psi (\omega)
$$
where $g$ is doubly stochastic, $\phi \in \cG_1$, and $\psi  \in \cG_2$.
\end{definition}

Again we postpone the examples to \S~\ref{ss.XY}.

Let us relate the scaling mean and Sinkhorn decompositions:

\begin{proposition}\label{p.func_smean_via_sink}
If $f$ has a Sinkhorn decomposition $\phi g \psi $ then
\begin{enumerate}
\item\label{i.sink_a} $\smean(f)= \gmean(\phi) \gmean(\psi )$.
\item\label{i.sink_b} The infimum in \eqref{e.def_func_sm} is attained at the functions $g_1 = 1/\phi$, $g_2 = 1/\psi$.
 \end{enumerate}
\end{proposition}

\begin{proof}
Part~(\ref{i.sink_a}) is an immediate consequence of Proposition~\ref{p.func_smean_ds} and the homogeneity property~\eqref{e.func_smean_homog}.
Part~(\ref{i.sink_b}) follows from part~(\ref{i.sink_a}).
\end{proof}

Our next result gives a sufficient condition for the existence of Sinkhorn decompositions.
We denote by $\logbounded(\P)$ the space of positive functions $f$ such that  $\log f \in L^\infty(\P)$.

\begin{theorem}\label{t.sink_for_bounded}
Every $f \in \logbounded(\P)$ has a Sinkhorn decomposition $\phi g \psi$ where $\phi \in \cB_1$ and $\psi \in \cB_2$.
Conversely, if $f = \phi' g' \phi'$ is another Sinkhorn decomposition such that $\phi' \in  \cB_1$ and $\psi' \in  \cB_2$ then there exists $\lambda\in \Rpos$ such that $\phi'= \lambda \phi$, $g' = g$, and $\psi'=\lambda^{-1}\psi$ $\P$-almost everywhere.
\end{theorem}

The proof of this theorem is independent of the rest of the paper and is presented in the  next subsection.
The theorem itself (in the particular case of directed products: see \S~\ref{ss.XY}) will be used in the proof of our Law of Large Permanents (Theorem~\ref{t.LLP}) in Section~\ref{s.LLP}.

\subsection{Proof of Theorem~\ref{t.sink_for_bounded}}\label{ss.hilbert}

The proof has two steps.
First, we reduce the problem to the existence of a fixed point for a certain nonlinear operator. Then we show that this operator contracts Hilbert's projective metric, and so it must have a fixed point. These ideas come from the literature. The fact that the existence of Sinkhorn decompositions for matrices is equivalent to a fixed point problem was first noted by Menon \cite{menon67}. The usefulness of Hilbert's projective metric in this context was noted independently in \cite{fl} (for matrices) and \cite{nu_book} (for matrices and for functions).
Unfortunately this elegant strategy uses strict positivity in an essential way, and in order to study more general functions other methods are needed: see \cite{nu,bln}.

\medskip

Fix a probability space $(\Omega, \cA, \P)$, a pair of sub-$\sigma$-algebras $\cA_1$, $\cA_2 \subset \cA$,
and a function $f \in \logbounded(\P)$.
Let $\cB_1$, $\cB_2 \subset \logbounded(\P)$ be as in \eqref{e.def_Bi}.
Define four maps forming a (non-commutative) diagram:
$$
\begin{tikzcd}
\cB_1 \arrow{r}{\mathtt{I}_1}  & \cB_1\arrow{d}{\mathtt{K}_2} \\
\cB_2 \arrow{u}{\mathtt{K}_1} & \cB_2 \arrow{l}{\mathtt{I}_2}
\end{tikzcd}
$$
by the formulas:
\begin{alignat*}{2}
(\mathtt{I}_i(\phi))(\omega)  &\coloneqq \frac{1}{\phi(\omega)}      \, , &\qquad
(\mathtt{K}_i(\phi))(\omega)  &\coloneqq \E(\phi f | \cA_i)(\omega)  \, .
\end{alignat*}
Let $\mathtt{T} \coloneqq \mathtt{K}_1 \circ \mathtt{I}_2 \circ \mathtt{K}_2 \circ \mathtt{I}_1$, i.e.,
the map obtained by going around the diagram.
Notice that $\mathtt{T}$ maps rays (i.e.\ sets of the form $\Rpos \phi$) into rays.
The following observation says that fixed rays yield Sinkhorn decompositions:

\begin{lemma}\label{l.fixed}
Suppose $\phi \in \cB_1$ and $c \in \Rpos$ are such that $\mathtt{T}(\phi) = c\phi$ $\P$-a.e.
Then $c=1$ and moreover, letting
$\psi \coloneqq \mathtt{K}_2 \circ \mathtt{I}_1 (\phi)$ and $g \coloneqq f/(\phi\psi)$,
the factorization $\phi g \psi$ is a Sinkhorn decomposition of $f$.
Conversely,  if $\phi g \psi$ is a Sinkhorn decomposition of $f$ with $\phi \in \cB_1$ and $\psi \in \cB_2$ then $\mathtt{T}(\phi) = \phi$  $\P$-a.e.
\end{lemma}

\begin{proof}
Suppose $\mathtt{T}(\phi) = c\phi$ $\P$-a.e., and let $\psi$ and $g$ be as above.
Then the following equalities hold $\P$-a.e.:
\begin{alignat*}{2}
\E(g | \cA_2)
&= \E\big(\tfrac{f}{\phi \psi} \big| \cA_2\big)  &\quad&\text{(by definition of $g$)}
\\
&= \tfrac{1}{\psi} \E\big(\tfrac{f}{\phi} \big| \cA_2\big) &\quad&\text{(by property~\eqref{e.condex_product} and the fact that $\psi \in \cB_2$)}
\\
&= 1 &\quad&\text{(by definition of $\psi$).}
\end{alignat*}
Similarly, the relation $\mathtt{K}_1 \circ \mathtt{I}_2 (\psi) = c \phi$ implies that
$$
\E(g | \cA_1) = c \quad \text{$\P$-a.e.}
$$
Integrating with respect to $\P$ yields that $c=1$,
and therefore $g$ is doubly stochastic, as we wanted to show.
The converse part of the lemma is immediate.
\end{proof}

Hence to complete the proof of Theorem~\ref{t.sink_for_bounded}
we are left to show that $\mathtt{T}$ has a fixed ray.
We will use a classical geometric--analytical device.

Given two functions $\phi$, $\hat\phi \in \cB_1$,
we define their \emph{Hilbert distance} as:
\begin{equation}\label{e.Hilbert}
d(\phi,\hat \phi) \coloneqq \log \left( \frac{\esssup(\hat\phi/\phi)}{\essinf(\hat \phi/\phi)}  \right) \, .
\end{equation}
This is a pseudometric such that $d(\phi,\hat \phi) = 0$ iff $\hat \phi = c \phi$ $\P$-a.e.\ for some constant $c \in \Rpos$.
In other words, if elements of $\cB_1$ that coincide a.e.\ are considered equal
then $d$ induces a genuine metric on the quotient space $\cB_1/\Rpos$.
Moreover, this metric is complete.
Analogously, we define a pseudometric on $\cB_2$, which we also denote by $d$.
Definition \eqref{e.Hilbert} is only a particular case of Hilbert's projective metric on convex cones; see e.g.\ \cite{nu_book} or \cite[Section~1]{Liverani}.
In these references the reader will find an important property due to Garrett Birkhoff \cite{b2},
which specialized to our case is stated as follows:

\begin{proposition}\label{p.tanh}
If $\mathtt{L} \colon \cB_1 \to \cB_2$
is linear and its image has finite diameter $\delta$
then the following contraction property holds:
$$
d(\mathtt{L}(\phi), \mathtt{L}(\hat \phi)) \le \left( \tanh \frac{\delta}{4} \right) \, d(\phi, \hat\phi)
\quad \text{for all } \phi, \hat \phi \in \cB_1 \, .
$$
\end{proposition}


\begin{proof}[Proof of Theorem~\ref{t.sink_for_bounded}]
The linear map $\mathtt{K}_2(\phi) \coloneqq \E(\phi  f| \cA_2)$
defined before satisfies the inequalities
$$
\left( \essinf f \right) \E(\phi | \cA_2)
\le \mathtt{K}_2(\phi) \le
\left( \esssup f \right)\E(\phi | \cA_2)
\quad \P\text{-a.e.}
$$
Analogous inequalities for $\hat \phi$ yield
\[
\left(\frac{\essinf f}{\esssup f}\right)\frac{ \E(\phi | \cA_2)}{ \E(\hat \phi | \cA_2)}\leq \frac{\mathtt{K}_2(\phi)}{\mathtt{K}_2(\hat \phi)}\leq \left(\frac{\esssup f}{\essinf f}\right)\frac{ \E(\phi | \cA_2)}{ \E(\hat \phi | \cA_2)} \, ,
\]
which by \eqref{e.Hilbert} imply that the image of $\mathtt{K}_2$ has diameter
$$
\delta \le 2 \log \left( \frac{\esssup f}{ \essinf f} \right) < \infty \, .
$$
So by the Proposition~\ref{p.tanh} the map $\mathtt{K}_2$ contracts Hilbert distances uniformly.
Analogously for $\mathtt{K}_1$.
On the other hand, the nonlinear maps $\mathtt{I}_1$ and $\mathtt{I}_2$ preserve Hilbert distances.
In particular $\mathtt{T} = \mathtt{K}_1 \circ \mathtt{I}_2 \circ \mathtt{K}_2 \circ \mathtt{I}_1$ contracts Hilbert distances uniformly.
By completeness we conclude that $\mathtt{T}$ has an unique fixed ray, so Lemma~\ref{l.fixed} allows us to conclude.
\end{proof}

As a byproduct of the proof above, the iterates under $\mathtt{T}$ of any initial ray converge with a known exponential rate to a fixed ray. This gives an effective algorithm for the computation of a Sinkhorn decomposition of $f$ and of its scaling mean, in particular.

\subsection{The main example: direct products}\label{ss.XY}

Here we describe the particular form of the previous results that will be used in Sections~\ref{s.LLP} and \ref{s.appl}.

Let $(X, \cX, \mu), (Y, \cY, \nu)$ be two probability spaces.
Take $(\Omega, \cA, \P)$ as the product space
$(X\times Y, \cX \times \cY, \mu\times \nu)$, and take the two sub $\sigma$-algebras on $\cA_1 = \cX \times \{\emptyset, Y\}$, $\cA_2=\{\emptyset, X\}\times \cY$.

An $\cA_1$-measurable function is just a function that only depends on the $x$ coordinate, and analogously for $\cA_2$-measurable functions.
We denote by $\logint(\mu)$, respectively $\logbounded(\mu)$,
the sets of positive measurable functions
$\phi \colon X \to \Rpos$ such that $\log \phi$ is
$\mu$-integrable, respectively, $\mu$-essentially bounded.
Analogously we define sets $\logint(\nu)$ and $\logbounded(\nu)$.
Recalling notations \eqref{e.def_Gi}, \eqref{e.def_Bi}, with a small abuse of language we can write
$$
\cG_1 = \logint(\mu),		\quad
\cB_1 = \logbounded(\mu),	\quad
\cG_2 = \logint(\nu),		\quad
\cB_2 = \logbounded(\nu).
$$
In this context, the previously introduced concepts can be recast as follows: 
\begin{itemize}
\item
The {scaling mean} of a measurable function $f\colon X \times Y \to \Rnon$ is:
\begin{equation}\label{e.sm_XY}
\smean(f) =  \inf_{\substack{\phi \in \logint(\mu) \\ \psi  \in \logint(\nu)}}
\frac{1}{\gmean(\phi) \gmean(\psi)} \iint \phi(x) f(x, y) \psi(y) \dd \mu(x) \dd\nu(y) \, .
\end{equation}
\item
The conditional expectations of a measurable function $f\colon X \times Y \to \R$ are:
\[
\E(f|\cA_1)=\smallint f(\cdot, y) \dd\nu(y) \  \textrm{ and } \ \E(f|\cA_2)=\smallint f(x, \cdot) \dd\mu(x).
\]
\item
A function $f\colon X \times Y \to \Rnon$ is {doubly stochastic} if:
\begin{equation}\label{e.ds_product}
\smallint f(\cdot, y) \dd \mu = \smallint f(x, \cdot) \dd \nu = 1
\quad
\text{for $\mu$-a.e.\ $x$ and $\nu$-a.e.\ $y$.}
\end{equation}
\item
A {Sinkhorn decomposition} of a function  $f\colon X \times Y \to \Rnon$
is a factorization of the form:
\begin{equation} \label{e.sink_product}
f(x, y) = \phi(x) g(x, y) \psi (y)
\end{equation}
where $g$ is doubly stochastic, $\phi \in \logint(\mu)$, and $\psi  \in \logint(\nu)$.
\end{itemize}

In the case that $X$ and $Y$ are finite sets and the probabilities $\mu$ and $\nu$ are equidistributed,
we reobtain the concepts studied in Section~\ref{s.2_means};
notice however that the definition of doubly stochastic matrices uses a different normalization.

Let us remark that functional Sinkhorn decompositions as in \eqref{e.sink_product} were first obtained by Knopp and Sinkhorn~\cite{ks} assuming that $X$ and $Y$ are compact and that $f$ is continuous and strictly positive.
For other existence results under various hypotheses, see \cite{nu,bln}.

\bigskip

Let us see a concrete non-trivial situation where functional Sinkhorn decompositions and scaling means can be computed directly. The result below will be used in Section~\ref{s.appl}.

\begin{proposition}\label{p.HS_smean}
Suppose that
$f \colon X \times Y \to \Rnon$ is a measurable nonnegative function of the form
$$
f(x,y) =
\begin{cases}
	f_0(x) &\text{if } y \in Y_0 \, , \\
	f_1(x) &\text{if } y \in Y_1 \, ,
\end{cases}
$$
where $\{Y_0,Y_1\}$ is a partition of $Y$ into sets of positive measure,
and $\log(f_0 + f_1)$ is $\mu$-integrable.
Then
\begin{equation}\label{e.HS_smean}
\smean(f) =
c^c \left(\frac{1-c}{r}\right)^{1-c} \exp \left( \int \log(f_0 + rf_1)\dd\mu\right) \, ,
\end{equation}
where $c \coloneqq \nu(Y_0)$
and $r$ is the unique positive  root of the equation
\begin{equation}\label{e.HS_smean_radius}
\int \frac{f_0}{f_0 + r f_1} \dd\mu = c \, .
\end{equation}
\end{proposition}

\begin{proof}
By dominated convergence, the LHS of \eqref{e.HS_smean_radius} is continuous and increasing with respect to $r$, and so by the intermediate value theorem the equation has a unique solution $r \in (0,\infty)$.
Let $\phi \coloneqq f_0 + r f_1$.
Note that
$\left|\log \phi - \log(f_0 + f_1)\right| \le \left| \log r \right|$,
and therefore $\phi \in \logint(\mu)$.
Let
$$
\psi(y) \coloneqq
\begin{cases}
	c 		&\text{if } y \in Y_0 \, , \\
	(1-c)/r	&\text{if } y \in Y_1 \, .
\end{cases}
$$
Direct calculation shows that the function $g\coloneqq f/(\phi\psi)$ is doubly stochastic,
so $\phi g \psi$ is a Sinkhorn decomposition of $f$.
Proposition~\ref{p.func_smean_via_sink}(\ref{i.sink_a}) then yields \eqref{e.HS_smean}.
\end{proof}

\section{A Law of Large Permanents}\label{s.LLP}

\subsection{Statement and comments}
In this section we prove our main result (Theorem~\ref{t.LLP}), which was already stated in the Introduction.
Let us recall the statement for the reader's convenience, and also fix some notation.

Assume that $(X,\cX,\mu)$, $(Y,,\cY,\nu)$ are Lebesgue probability spaces,
and $T \colon X \to X$, $S \colon Y \to Y$ are measure preserving transformations.
Given a function $f\colon X \times Y \to \R$,
for each $(x,y) \in X\times Y$
and each positive integer $n$ we define the following $n \times n$ matrix:
\begin{equation}\label{e.dyn_matrix}
\dmat{f}{n} (x,y) \coloneqq
\begin{pmatrix*}[l]
f(x,y)        & f(Tx,y)        & \cdots & f(T^{n-1}x,y)        \\
f(x,Sy)       & f(Tx,Sy)       & \cdots & f(T^{n-1}x,Sy)       \\
\qquad\vdots  & \qquad\vdots   &        & \qquad\vdots         \\
f(x,S^{n-1}y) & f(Tx,S^{n-1}y) & \cdots & f(T^{n-1}x,S^{n-1}y)
\end{pmatrix*},
\end{equation}
which can be thought as the truncation of an infinite matrix $\dmat{f}{}(x,y) \in \Mat{\N}$.

Recall that $\logbounded (\mu \times \nu)$ denotes the set
of positive measurable functions on $X \times Y$ essentially bounded away from zero and infinity.
Such functions have scaling means, given by formula  \eqref{e.sm_XY}.
We can now restate our main result as follows: 

\begin{theorem} \label{t.LLP}
If $T$ and $S$ are ergodic, and $f \in \logbounded (\mu \times \nu)$
then
\begin{equation} \label{e.LLP}
\lim_{n \to \infty} \pmean \left( \dmat{f}{n}(x, y) \right) = \smean(f)
\end{equation}
for $\mu \times \nu$-almost every $(x, y) \in X \times Y$.
\end{theorem}


This \emph{Law of Large Permanents} is a very general ergodic theorem.
In Section~\ref{s.appl} we will see how to apply it to other natural types of means,
and in Section~\ref{s.questions} we will discuss the possibility of even more general laws.
We stress that Theorem~\ref{t.LLP} not only states the existence of the limit in \eqref{e.LLP}, but characterizes it as a scaling mean, which can be used to efficiently  compute its value.

\smallskip

It is worthwhile to note that Theorem~\ref{t.LLP} implies the generalized Friedland limit (Theorem~\ref{t.friedland}), at least for positive matrices. Indeed, given $A\in \Map{k}$, take $T$ and $S$ as cyclic permutations on sets $X$ and $Y$ of cardinality $k$, and let $\mu$ and $\nu$ be the corresponding invariant measures. Let $f$ be such that $\dmat{f}{k}(x,y) = A$ for some point $(x,y)$. Then, for every $m \ge 1$, the matrix $\dmat{f}{km}(x,y)$ is permutationally equivalent to the Kronecker product $A \otimes U_m$, where $U_m$ is the $m \times m$ matrix all of whose entries are $1$.
In particular, $\pmean(A \otimes U_m)$ equals $\pmean(\dmat{f}{km}(x,y))$, which by
Theorem~\ref{t.LLP} converges to $\smean(f)$ as $k \to \infty$.
Using for instance a Sinkhorn decomposition of the matrix $A$ and the related Sinkhorn decomposition of the function~$f$, one checks that $\smean(f) = \smean(A)$, thus obtaining the generalized Friedland limit~\eqref{e.generalized_friedland}.

Actually, we will reason in the converse direction and deduce Theorem~\ref{t.LLP}
from Theorem~\ref{t.friedland}.
Let us sketch the proof.
The first step is to approximate $f$ by a suitable simple function, and to show that the permanental means do not change much; the values of the simple function are recorded on a square matrix $A$.
The second step is to show that the matrix \eqref{e.dyn_matrix} is, modulo a permutation of rows and columns, approximately equal to a Kronecker product $A \otimes U_m$, and then to use Theorem~\ref{t.friedland} to relate the permanental means of these matrices with scaling means.
It is technically convenient to work with doubly stochastic functions, so we will also use Theorem~\ref{t.sink_for_bounded} on the existence of Sinkhorn decompositions.

\smallskip

The precise proof of Theorem~\ref{t.LLP} will take up the rest of this section,
which is organized as follows:
In \S~\ref{ss.prelim} we recall some basic results: approximation by conditional expectations, and an ergodic theorem.
In \S~\ref{ss.reg_per} we obtain some estimates on how the permanent changes under perturbations of the matrix.
These facts are used in \S~\ref{ss.LLP_proof} to implement the strategy sketched above and to prove the theorem.

\subsection{Preliminaries from Measure Theory and Ergodic Theory}\label{ss.prelim}

We have defined conditional expectations in \S~\ref{ss.condex_ds}.
The following result is contained in \cite[Theorems~10.2.2, 10.2.3]{boga}
and describes the behavior of conditional expectations as the $\sigma$-algebra
is refined:

\begin{proposition}\label{p.condex_refine}
Let $(\Omega, \cA , \mu)$ be a probability space and let $f \in L^1(\mu)$.
Suppose that $(\cA_k)$ is an increasing sequence of sub-$\sigma$-algebras of $\cA$
whose union generates the $\sigma$-algebra $\cA$ modulo sets of measure zero.
Then the functions $\E(f | \cA_k)$ converge
almost surely and in $L^1$ to $f$ as $k \to \infty$.
\end{proposition}

Let us describe the basic ergodic theorem that we will need.
Consider probability spaces $(X,\mu)$, $(Y,\nu)$
and measure preserving transformations
$T \colon X \to X$, $S \colon Y \to Y$.
Then we can define a measure preserving action of the semigroup $\N^2$
(where $0 \in \N$)
on the product space $(X \times Y,\mu \times \nu)$ as follows:
$$
\cT^{(i,j)} (x,y) \coloneqq (T^i x, S^j y), \quad
\text{where $(i,j) \in \N^2$ and $(x,y) \in X\times Y$.}
$$
Notice that this action is ergodic if and only if both $T$ and $S$ are ergodic.
In this case, by the ergodic theorem for $\N^2$-actions
(see \cite[Theorem 2.1.5]{ke} or \cite[Chapter 6, Theorem 3.5]{krengel}),
for any $h \in L^1(\mu \times \nu)$ we have
\begin{equation}\label{e.ergodic_thrm}
\lim_{n \to \infty} \frac{1}{n^2} \sum_{i,j =0}^{n-1} h(T^i x, S^j y)
= \int h \dd(\mu \times \nu) \qquad
\text{for $\mu\times\nu$-a.e.\ }(x,y).
\end{equation}

\subsection{More preliminaries: Regularity estimates for the permanent of strictly positive matrices}\label{ss.reg_per}


Let us begin by recalling some basic facts.
As for the determinant, the permanent of a square matrix $A = (a_{ij}) \in \Mat{n}$
can be computed by means of a
\emph{Laplace expansion} along any column $j$:
$$
\per A = \sum_{i=1}^{n} a_{ij} \per A(i|j),
$$
where, as usual, $A(i|j)$ denotes the matrix obtained from $A$ by deleting the $i$-th row and the $j$-th column.
Similar Laplace expansions along rows hold.
Using either kind of expansion, we see that the partial derivatives of
the permanent function are simply:
\begin{equation}\label{e.partial}
\frac{\partial \per A}{\partial a_{ij}} = \per A(i|j).
\end{equation}

Given $\lambda > 1$, let us say that a positive matrix $A =(a_{ij}) \in \Map{n}$
is \emph{$\lambda$-bounded} if:
\begin{equation}\label{e.lambda_bounded}
\lambda^{-1} \le a_{ij} \le \lambda \quad \text{for each $i$, $j$.}
\end{equation}

\begin{lemma}\label{l.minors}
Let $A \in \Map{n}$ be a $\lambda$-bounded matrix and denote by

 $s$ 
be the sum of the entries of $A$.
Then for any $i$, $j \in \{1,\dots,n\}$ we have
$$
\lambda^{-5} n \per A \le
\lambda^{-4} n s^{-1} \per A \le \per A(i|j) \le \lambda^4 n s^{-1} \per A
\le \lambda^5 n \per A \, .
$$
\end{lemma}

\begin{proof}
The outer inequalities being trivial, we only need to care about the inner ones.
Summing Laplace expansions of the permanent of $A=(a_{ij})$ along the $n$ columns,
we obtain $n \per A = \sum_{i,j} a_{ij} \per A(i|j)$, and in particular,
$$
s \min_{i,j} \per A(i|j) \le n \per A \le s \max_{i,j} \per A(i|j) \, .
$$	
Suppose that the minimum and the maximum above are attained at
$(i_1,j_1)$ and $(i_2,j_2)$, respectively.
Consider the two Laplace expansions:
\begin{align*}
\per A(i_2|j_2) &= \sum_{k \neq i_2} a_{kj_1} \per A(i_2 k | j_1 j_2) \, , \\
\per A(i_2|j_1) &= \sum_{k \neq i_2} a_{kj_2} \per A(i_2 k | j_1 j_2) \, .
\end{align*}
Since $a_{kj_1} \le \lambda^2 a_{kj_2}$, it follows that
$A(i_2|j_2) \le \lambda^2 \per A(i_2|j_1)$.
An analogous argument gives $\per A(i_2|j_1) \le \lambda^2 \per A(i_1|j_1)$,
thus proving that $\per A(i_2|j_2) \le \lambda^4 \per A(i_1|j_1)$.
So for any $i$, $j$ we have
$$
\per A(i|j) \le \per A(i_2|j_2) \le \lambda^4 \per A(i_1|j_1)
\le \lambda^4 n s^{-1} \per A \, ,
$$
as claimed.
The remaining inequality is proved analogously.
\end{proof}

We can now prove a regularity estimate for the permanent:

\begin{lemma}\label{l.Lip}
If $A$, $B \in \Map{n}$ are $\lambda$-bounded matrices then
$$
\left| \log \frac{\per B}{\per A}\right| \le
\frac{\lambda^4 n \sum_{i,j} |b_{ij} - a_{ij}|}{\min \left\{\sum_{i,j} a_{ij} , \ \sum_{i,j} b_{ij} \right\}}
\le \frac{\lambda^5}{n} \sum_{i,j} |b_{ij} - a_{ij}|
\, .
$$
\end{lemma}

\begin{proof}
The second inequality being trivial, we only need to care about the first one.
For all $t \in [0,1]$, the convex combination
$A_t \coloneqq (1-t)A + tB$ is also a positive  $\lambda$-bounded matrix.
To prove the first inequality, we will apply the mean value theorem to the function
$f(t) \coloneqq \log \per A_t$ and use the estimate
\begin{equation}\label{e.mvt}
\left| \log \frac{\per B}{\per A}\right| = |f(1) - f(0)| \le \max_{t\in [0,1]} |f'(t)| \, .
\end{equation}
Using formula~\eqref{e.partial}, we compute:
$$
f'(t) = \frac{1}{\per A_t} \sum_{i,j} (b_{ij} - a_{ij}) \per A_t(i|j),
$$
while by Lemma~\ref{l.minors} we have
$$
\frac{\per A_t(i|j)}{\per A_t} \le \frac{\lambda^4 n}{\sum_{i,j} (1-t) a_{ij} + t b_{ij}} \, .
$$
Plugging these estimates into \eqref{e.mvt} we obtain the desired inequality.
\end{proof}

As another consequence of Lemma~\ref{l.minors},
we obtain the following estimate on how the permanental mean
of a matrix varies when a row and a column are deleted:

\begin{lemma}\label{l.pmean_delete}
There is a constant $C>1$ such that
for any $\lambda>1$ and any $n \ge 2$,
if  $A \in \Map{n}$ is a $\lambda$-bounded matrix
then for any $i$, $j \in \{1,\dots,n\}$ we have
\begin{equation*} 
\left| \log \frac{\pmean(A(i|j))}{\pmean(A)} \right| \le
\frac{6 \log \lambda + C \log n}{n} \, .
\end{equation*}
\end{lemma}

\begin{proof}
By Stirling's formula, $\log(n!) = n \log n - n + \cO(\log n)$, and so
$$
\log \pmean(A) = \frac{1}{n}\log \per A + \log n - 1 + \cO\left( \frac{\log n}{n} \right) \, .
$$
Analogously, letting $B \coloneqq A(i|j)$,
$$
\log \pmean(B) = \frac{1}{n-1}\log \per B + \log n - 1 + \cO\left( \frac{\log n}{n} \right) \, .
$$
These two estimates yield:
\begin{align*}
\left| \log \frac{\pmean(B)}{\pmean(A)} \right|
&\le \left| \frac{1}{n}\log \pmean(B) \right| + \left| \frac{n-1}{n}\log \pmean(B)  - \log \pmean(A) \right|
\\
&\le
{\underbrace{\left| \frac{1}{n}\log \pmean(B) \right|}_{(\star)}} +
{\underbrace{\left| \frac{1}{n}\log \frac{\per B}{\per A} \right|}_{(\star\star)}} +
\cO\left( \frac{\log n}{n} \right) \, .
\end{align*}
Since $B$ is $\lambda$-bounded, it follows from internality of the permanental mean that
$(\star) \le (\log \lambda)/n$.
On the other hand, by Lemma~\ref{l.minors} we have
$(\star\star) \le (5\log \lambda + \log n)/n$.
The lemma follows.
\end{proof}


\subsection{Proof of Theorem~\ref{t.LLP}}\label{ss.LLP_proof}

\setcounter{subsubsection}{-1} 

\subsubsection{Zeroth step: Some reductions} \label{sss.zero}

In order to prove Theorem~\ref{t.LLP}, it is sufficient to
consider functions that, in addition of being in $\logbounded(\mu\times \nu)$,
are doubly stochastic, i.e., satisfy relations~\eqref{e.ds_product}.
Indeed, assuming the theorem already proved in this case, and given an
arbitrary $f \in \logbounded(\mu\times \nu)$,
we consider the Sinkhorn decomposition $f (x,y) = \phi(x) g(x,y) \psi(y)$ as in \eqref{e.sink_product}
given by Theorem~\ref{t.sink_for_bounded}.
Then, by row-wise and column-wise homogeneity of the permanental mean,
$$
\pmean \left( \dmat{f}{n}(x, y) \right)
= \left( \prod_{i=0}^{n-1} \phi(T^i x)\right)^{1/n} \left( \prod_{j=0}^{n-1} \psi(S^j y)\right)^{1/n} \pmean \left( \dmat{g}{n}(x, y) \right) \, .
$$
Since $T$ is ergodic, by Birkhoff's theorem the first factor on the RHS converges to
the geometric mean \eqref{e.func_gmean} for $\mu$-a.e.\ $x \in X$.
Analogously for the second factor.
Since we are assuming Theorem~\ref{t.LLP} already proved in the doubly stochastic case,
the third factor converges to the scaling mean of $g$, which by Proposition~\ref{p.func_smean_ds} equals $1$.
In conclusion, $\pmean (\dmat{f}{n})$ converges a.e.\ to $\gmean(\phi) \gmean(\psi)$,
which by Proposition~\ref{p.func_smean_via_sink}(\ref{i.sink_a}) equals $\smean(f)$.

So from this point on we assume that
\begin{equation}\label{e.hyp_ds}
f \in \logbounded(\mu\times \nu) \text{ is doubly stochastic,}
\end{equation}
and our aim is to show that the permanental means
of the matrices \eqref{e.dyn_matrix} converge a.e.\ to $1$.

\smallskip

As a second reduction, it is sufficient to prove Theorem~\ref{t.LLP} assuming that the measures $\mu$ and $\nu$ are non-atomic.
Indeed, fix an arbitrary non-atomic Lebesgue probability space $(Z,\theta)$ and an ergodic measure preserving transformation $U \colon Z \to Z$, and consider the product spaces
$(\hat{X}, \hat{\mu}) \coloneqq (X \times Z, \mu \times \theta)$ and
$(\hat{Y}, \hat{\nu}) \coloneqq (Y \times Z, \nu \times \theta)$.
Then the transformations
$\hat{T} (x,z) \coloneqq (Tx,Uz)$ and
$\hat{S} (y,w) \coloneqq (Sy,Uw)$ preserve the measures $\hat{\mu}$ and $\hat{\nu}$, respectively, and are ergodic.
Given a doubly stochastic function
$f \in (\log L^\infty)(\mu \times \nu)$, we consider the doubly stochastic function
$\hat{f}(x,z,y,w) \coloneqq f(x,y)$ on $\hat{X} \times \hat{Y}$.
Since the measures $\hat{\mu}$ and $\hat{\nu}$ are non-atomic, and assuming that Theorem~\ref{t.LLP} is already proved in this case, we conclude that the permanental mean of the matrix $\dmat{\hat{f}}{n}(x,z,y,w)$ converges to $1$ at $\hat{\mu}\times\hat{\nu}$-a.e.\ $(x,z,y,w)$. But the latter matrix obviously equals $\dmat{f}{n}(x,y)$, so we obtain \eqref{e.LLP}.

So we can assume that the Lebesgue spaces $(X,\mu)$ and $(Y,\nu)$ are non-atomic. Actually, for  convenience in the following proof, we will actually assume that
\begin{equation}\label{e.hyp_Leb}
X = Y = [0,1] \quad \text{and} \quad
\mu=\nu \text{ is Lebesgue measure.}
\end{equation}

\subsubsection{First step: Discretizing $f$}

Fix an arbitrary $\epsilon > 0$.
Let us construct a convenient discretized approximation of the given function $f$.

Let $k$ be a large positive integer (to be specified later),
and define a positive matrix $A = (a_{pq}) \in \Map{k}$ by
\[
a_{pq} \coloneqq k^2 \int_{(p-1)/k}^{p/k} \int_{(q-1)/k}^{q/k} f(x,y) \dd y \dd x \, .
\]
As a consequence of \eqref{e.hyp_ds},
the matrix $k^{-1}A$ is doubly stochastic.
Indeed, for every $p \in \{1,\dots,k\}$ the sum of the corresponding row of $A$ is
$$
\sum_{q=1}^k a_{pq}
= k^2 \int_{(p-1)/k}^{p/k} {\underbrace{\int_{0}^{1} f(x,y) \dd y}_{1}} \dd x
= k \, ,
$$
and in an analogous way we compute the column sums.
In particular, by the homogeneity of the scaling mean and Proposition~\ref{p.smean_ds},
we have $\smean(A) = 1$.

\smallskip

Let $g$ be a function on $X \times Y = [0,1]^2$
equal to the constant $a_{pq}$ on each sub-square
$[(p-1)/k, p/k) \times [(q-1)/k,q/k)$, where $p$, $q \in \{1,\dots,k\}$.
Notice that $g$ is nothing but the conditional expectation of $f$
with respect to the $\sigma$-algebra generated by the partition into these sub-squares.
Therefore, by Proposition~\ref{p.condex_refine}, if $k$ is chosen large enough then
$f$ and $g$ are $L^1$-close in the sense that
\begin{equation} \label{e.g_close}
\int |f-g| \dd(\mu\times\nu)  < \epsilon.
\end{equation}
We fix the integer $k$ and therefore the matrix $A$ and the function $g$
from now on.

\smallskip

Since $f \in \logbounded(\mu\times \nu)$,
there exists $\lambda>0$ such that
$$
\lambda^{-1} \leq f \leq \lambda
\qquad \mu\times\nu\text{-a.e.}
$$
The values of $g$ are obtained by averaging the values of $f$
and therefore satisfy the same bounds, that is,
$\lambda^{-1} \leq g  \leq \lambda$.
In particular, the matrices
$\dmat{f}{n}(x,y)$ and $\dmat{g}{n}(x,y)$ are $\lambda$-bounded
for a.e.\ $(x,y)$.
We now use Lemma~\ref{l.Lip} to compare their permanental means:
$$
\left| \log \frac{\pmean (\dmat{f}{n}(x,y))}{\pmean(\dmat{g}{n}(x,y))}\right|
=
\frac{1}{n} \left| \log \frac{\per (\dmat{f}{n}(x,y))}{\per(\dmat{g}{n}(x,y))}\right|
\le
\frac{\lambda^5}{n^2} \sum_{i,j=0}^n |f(T^i x, S^j y) - g(T^i x, S^j y)| \, .
$$
Using the ergodic theorem \eqref{e.ergodic_thrm} and the bound \eqref{e.g_close},
we conclude that for a.e.\ $(x,y)$,
\begin{equation}\label{e.limsup}
\limsup_{n\to \infty}
\left| \log \frac{\pmean (\dmat{f}{n}(x,y))}{\pmean(\dmat{g}{n}(x,y))}\right|
< \lambda^5 \epsilon \, .
\end{equation}

\subsubsection{Second step: Comparing $\dmat{g}{n}$ with a Kronecker product}

Let us fix $(x,y)$ such that \eqref{e.limsup} holds,
the $T$-orbit of $x$ visits each of the intervals
$[0,1/k]$, $[1/k, 2/k]$, \dots, $[(k-1)/k,1]$ with limit frequency $1/k$,
and analogously for the $S$-orbit of $y$.

For each $n \ge 1$,
since the points $x$ and $y$ are not periodic,
there exist permutation $\tau_n$ and $\sigma_n$
of the set $\{0,1,\dots, n-1\}$
such that
$$
T^{\tau_n(0)}x   < T^{\tau_n(1)}x   < \cdots < T^{\tau_n(n-1)}x   \, , \qquad
S^{\sigma_n(0)}y < S^{\sigma_n(1)}y < \cdots < S^{\sigma_n(n-1)}y \, .
$$
Let $P_{\tau_n}$ be the $n\times n$ permutation matrix
whose $1$'s are located on the positions $(i,\tau(i))$, $i \in \{0,1,\dots, n-1\}$.
Analogously define $P_{\sigma_n}$.
Consider the product matrix
$$
B_n \coloneqq P_{\tau_n} \cdot \dmat{g}{n}(x,y) \cdot P_{\sigma_n}^{-1} =
\big( g(T^{\tau_n(i)}x, S^{\sigma_n(j)} y) \big)_{i,j=0}^{n-1} \, .
$$
This matrix has a decomposition into blocks:
$$
B_n =
\begin{pmatrix}
B_{n,11} & B_{n,12} & \cdots & B_{n,1k} \\
B_{n,21} & B_{n,22} & \cdots & B_{n,2k} \\
\vdots       & \vdots       &        & \vdots       \\
B_{n,k1} & B_{n,k2} & \cdots & B_{n,kk}
\end{pmatrix}
$$
where the block $B_{n,pq}$ has all its entries are equal to $a_{pq}$,
has width equal to the cardinality of $\{x, Tx, \cdots, T^{n-1}x\} \cap [(p-1)/k,p/k]$,
and has height equal to the cardinality of $\{y, Sy, \cdots, S^{n-1}y\} \cap [(q-1)/k,q/k]$.
It follows that if $n = km$ for some sufficiently large integer $m$
then the matrix $B_n$ and the Kronecker product $A \otimes U_m$
differ at at most $\epsilon n$ rows and columns.
Since both matrices are $\lambda$-bounded, by Lemma~\ref{l.Lip} we have
$$
\left| \log \frac{\pmean(B_n)}{\pmean(A\otimes U_m)}\right|
=
\frac{1}{n} \left| \log \frac{\per(B_n)}{\per(A\otimes U_m)}\right|
\le
\frac{1}{n} \, \frac{\lambda^5}{n} \, (\lambda - \lambda^{-1}) \, \epsilon n^2
< \lambda^6 \epsilon \, .
$$
Since the matrices $B_n$ and $\dmat{g}{n}(x,y)$ are by definition permutationally equivalent,
they have the same permanental mean (by the row-/column-wise symmetry property).
On the other hand, by the generalized Friedland limit (Theorem~\ref{t.friedland}),
we have
$$
\lim_{m \to \infty} \pmean(A \otimes U_m) = \smean(A) = 1 \, .
$$
So we obtain
$$
\limsup_{m \to \infty} \left| \log \pmean(\dmat{g}{km}(x,y)) \right| \le \lambda^6 \epsilon \, .
$$
Notice that, as a consequence of Lemma~\ref{l.pmean_delete},
$$
\lim_{n \to \infty} \log
\frac{\pmean(\dmat{\, g}{k \lfloor n/k \rfloor}(x,y))}{\pmean(\dmat{g}{n}(x,y))} = 0 \, ,
$$
and so
$$
\limsup_{n \to \infty} \left| \log \pmean(\dmat{g}{n}(x,y)) \right| \le \lambda^6 \epsilon \, .
$$
Recalling \eqref{e.limsup}, we obtain
$$
\limsup_{n \to \infty} \left| \log \pmean(\dmat{f}{n}(x,y)) \right| \le (\lambda^6+\lambda^5) \epsilon \, .
$$
Since $\epsilon>0$ is arbitrary, we infer that
$$
\lim_{n \to \infty} \pmean(\dmat{f}{n}(x,y)) = 1 \, .
$$
This proves \eqref{e.LLP} under the assumptions
\eqref{e.hyp_ds} and \eqref{e.hyp_Leb}.
As explained in \S~\ref{sss.zero}, Theorem~\ref{t.LLP} in full generality follows.

%

\section{Applications}\label{s.appl}

\subsection{Symmetric means} \label{ss.sym}

The \emph{elementary symmetric polynomial} of degree $k$ in $n \ge k$
variables is defined as
\[
E_k(z_1,z_2,\dots,z_n) \coloneqq \sum_{1 \leq i_1 < i_2 < \dots <i_k \leq n} z_{i_{1}}z_{i_{2}} \cdots z_{i_{k}}.
\]
These sums appears in a wide range of different areas of mathematics.
For example, Vieta's formula states that if $P(z) =  z^n + a_{n-1} z^{n-1}+ \dots + a_1z +a_0$ is a monic polynomial and
$z_1$, \dots , $z_n$ are its roots (repeated according to multiplicity)
then $a_{n-k} = (-1)^k E_k(z_1,\dots,z_n)$.

Assuming that $z=(z_1, \dots, z_n)$ is a string of nonnegative real numbers,
we define its \emph{$k$-th symmetric mean} as
\[
\symean_k (z) \coloneqq
\left( \frac{E_k(z_1,\dots, z_n)}{\left( {n \atop k } \right)}  \right)^{1/k} \, .
\]
Notice that $\symean_1(z)$ is the arithmetic mean, and $\symean_n(z)$ is the geometric mean.
In the 18th century, MacLaurin discovered that $\symean_1(z) \ge \symean_2(z) \ge \cdots \ge \symean_n(z)$, thus generalizing the AM--GM inequality (see \cite[\S~2.22]{hlp}).

Symmetric means have the properties of
reflexivity, monotonicity, internality, continuity, homogeneity, and, of course, symmetry.
Actually, they can be expressed in terms of permanental means:
we have
\begin{equation}\label{e.symean_pmean}
\symean_k(z) = \big[ \pmean ( R_k(z) ) \big]^{n/k} \, ,
\end{equation}
where
\newlength{\shiftdistance}\settowidth{\shiftdistance}{$n-k$ rows}\addtolength{\shiftdistance}{.8cm}
\begin{equation}\label{e.rep_matrix}
R_k(z_1,\dots,z_n) \coloneqq
\left(
\begin{tikzpicture}[baseline=-1.35cm,x=1cm,y=1cm]
\draw(0,0)    node{$z_1$};   \draw(1,0)    node{$z_2$};   \draw(2,0)   node{$\cdots$};\draw(3,0)   node{$z_n$};
\draw(0,-.4)  node{$\vdots$};\draw(1,-.4)  node{$\vdots$};                            \draw(3,-.4) node{$\vdots$};
\draw(0,-1)   node{$z_1$};   \draw(1,-1)   node{$z_2$};   \draw(2,-1)  node{$\cdots$};\draw(3,-1)  node{$z_n$};
\draw(0,-1.5) node{$1$};     \draw(1,-1.5) node{$1$};     \draw(2,-1.5)node{$\cdots$};\draw(3,-1.5)node{$1$};
\draw(0,-1.9) node{$\vdots$};\draw(1,-1.9) node{$\vdots$};                            \draw(3,-1.9)node{$\vdots$};
\draw(0,-2.5) node{$1$};     \draw(1,-2.5) node{$1$};     \draw(2,-2.5)node{$\cdots$};\draw(3,-2.5)node{$1$};
\draw[thick,decorate,decoration={brace}] (3.75,.15) -- (3.75,-1.15) node[midway,right,xshift=.1cm]{$k$ rows};
\draw[thick,decorate,decoration={brace}] (3.75,-1.35) -- (3.75,-2.65) node[midway,right,xshift=.1cm]{$n-k$ rows};
\end{tikzpicture}
\hspace{-\shiftdistance}
\right)
\hspace{\shiftdistance}
\end{equation}

Using this relation, we can deduce from the Law of Large Permanents (Theorem~\ref{t.LLP})
an ergodic theorem for symmetric means.
Actually, such result already exists and was obtained in 1976 by Hal\'asz and Sz\'ekely~\cite{HS}:

\begin{theorem}[Hal\'asz--Sz\'ekely] \label{t.HS}
Let $(X,\mu)$ be a Lebesgue probability space,
$T \colon X \to X$ be a measure-preserving ergodic transformation,
$g \in \logbounded(\mu)$,
and $0 < c < 1$.
Suppose $k(n)$ is a sequence of integers satisfying
$$
1 \le k(n) \le n \quad \text{and} \quad \lim_{n \to \infty} \frac{k(n)}{n} = c \, .
$$
Then, for $\mu$-almost every $x$, the limit
$$
\lim_{n \to \infty} \symean_{k(n)}\big( g(x), g(Tx), \dots, g(T^{n-1} x) \big)
$$
exists and equals
\begin{equation} \label{e.HS}
\symean_c(g) \coloneqq
c \left(\frac{1-c}{r}\right)^{\frac{1-c}{c}} \exp \left( \frac{1}{c}\int \log(g + r) \dd\mu\right) \, ,
\end{equation}
where $r = r(c)$ is the unique positive root of the equation
\begin{equation*} 
\int \frac{g}{g + r} \dd\mu =  c \, .
\end{equation*}
\end{theorem}

\begin{proof}
Let $Y = \{0,1\}^\N$, let $\nu$ be the Bernoulli measure with weights $c$, $1-c$, let $S \colon Y \to Y$ be the shift, and let $\{Y_0,Y_1\}$ the partition of $Y$ into the cylinders of length $1$.
Consider the function
$$
f(x,y) \coloneqq
\begin{cases}
	g(x) &\text{if } y \in Y_0 \, , \\
	1    &\text{if } y \in Y_1 \, .
\end{cases}
$$
For $\mu$-a.e.\ $x$ and $\nu$-a.e.\ $y$, the conclusion \eqref{e.LLP} of Theorem~\ref{t.LLP} holds,
and moreover if  $\ell(n)$ denotes the cardinality of the set $\{y, Sy, \dots, S^{n-1}y\} \cap Y_0$
then $\ell(n) / n \to c$ as $n \to \infty$.
Using notation~\eqref{e.rep_matrix}, define matrices
$$
A_n \coloneqq R_{\ell(n)}(g(x),\dots,g(T^{n-1}x)) \, , \qquad
B_n \coloneqq R_{k(n)}(g(x),\dots,g(T^{n-1}x)) \, .
$$
Notice that $A_n$ can be obtained from $\dmat{f}{n}(x,y)$ by permuting rows.
In particular,
$$
\lim_{n \to \infty} \pmean(A_n) = \lim_{n \to \infty} \pmean(\dmat{f}{n}(x,y)) = \smean(f) \, .
$$
Let $\lambda>1$ be such that $\lambda^{-1} \le g \le \lambda$.
Then the matrices $A_n$ and $B_n$ are $\lambda$-bounded, and so by Lemma~\ref{l.Lip},
$$
\left| \log \frac{\pmean(B_n)}{\pmean(A_n)}\right|
=
\frac{1}{n} \left| \log \frac{\per(B_n)}{\per(A_n)}\right|
\le
\frac{1}{n} \, \frac{\lambda^5}{n} \, (\lambda - \lambda^{-1}) \, n |\ell(n) - k(n)|
\le \frac{\lambda^6 |\ell(n) - k(n)|}{n} \, ,
$$
which converges to $0$ as $n \to \infty$.
So $\pmean(B_n) \to \smean(f)$,
and therefore by \eqref{e.symean_pmean} we conclude that
$\symean_{k(n)} (g(x),\dots, g(T^{n-1}x)) \to [\smean(f)]^{1/c}$.
We conclude the proof  using Proposition~\ref{p.HS_smean} to compute $\smean(f)$.
\end{proof}


Let us compare the result above with that  of Hal\'asz and Sz\'ekely's paper \cite{HS}.
The theorem from that paper requires only a weak integrability condition, namely, $\int \log(1+g)\dd\mu<\infty$.
The theorem is stated in terms of independent identically distributed random variables,
but the proof actually does not use independence, and ergodicity suffices.
Therefore, the actual Hal\'asz--Sz\'ekely theorem is stronger than Theorem~\ref{t.HS} above.
This indicates that a weakening of the hypotheses of Theorem~\ref{t.LLP} should be pursued
(more about this on Section~\ref{s.questions} below) and
should not be regarded as an inherent drawback of our approach.
Let us mention that the proof in \cite{HS} employs completely different tools, namely: Vieta's formula and Cauchy integral formula are used to relate the means with a certain complex integral, and then the saddle point method is used to estimate the value of the integral.
This line of argument has been used in most of the probability papers on the subject.
Our methods, on the other hand, provide a more transparent explanation for the complicated formula \eqref{e.HS},
and apply to much more general types of means.
A second example of application is given in the next subsection.

\subsection{Muirhead means} \label{ss.Muirhead}

Let $z=(z_1, \dots, z_n)$  and $\alpha=(\alpha_1, \dots, \alpha_n)$
be two strings of nonnegative numbers, not all the $\alpha_i$'s being $0$.
We then define the \emph{Muirhead $\alpha$-mean} of the $z_i$'s as
$$
\mumean_\alpha(z) \coloneqq \left( \frac{1}{n!}\sum_{\sigma \in S_n} \prod_{i=1}^m z_{\sigma(i)}^{\alpha_i} \right)^{\frac{1}{\alpha_1 + \dots + \alpha_n }} \, ,
$$
where $S_n$ denotes the set of permutations of $\{1, \dots, n\}$.
Note that this coincides with the arithmetic mean if $\alpha=(1,0,\dots,0)$,
and with the geometric mean if $\alpha = (1,1,\dots,1)$.
Muirhead means also generalize the symmetric ones; indeed
$$
\symean_k(z) = \mumean_{\alpha_k}(z) \quad \text{where} \quad
\alpha_k \coloneqq \big( {\underbrace{1, 1, \dots, 1}_{k}}, \, {\underbrace{0, 0, \dots, 0}_{n-k}} \big) \, .
$$

The celebrated Muirhead--Hardy--Littlewood--P\'olya inequality \cite[\S~2.18]{hlp}
describes when the functions $\mumean_{\alpha}(\mathord{\cdot})$ and $\mumean_{\beta}(\mathord{\cdot})$
are comparable; this is done in terms of a concept called \emph{majorization},
which appears in a vast number of other situations.

\smallskip

Muirhead means can also be expressed in terms of permanental ones:
$$
\mumean_\alpha(z) = [\pmean(M_\alpha(z))]^{\frac{n}{\alpha_1+\cdots+\alpha_n}}
\quad \text{where} \quad
M_\alpha(z) \coloneqq
\begin{pmatrix}
z_1^{\alpha_1} & \cdots & z_n^{\alpha_1} \\
\vdots         &        & \vdots         \\
z_1^{\alpha_n} & \cdots & z_n^{\alpha_n}
\end{pmatrix} \, .
$$
Therefore we can deduce from the Law of Large Permanents  an ergodic theorem for Muirhead means,
namely:

\begin{corollary}
Let $(X,\mu)$, $(Y,\nu)$ be a Lebesgue probability spaces,
$T \colon X \to X$, $S\colon Y \to Y$ be measure-preserving ergodic transformations,
and $g \in \logbounded(\mu)$, $h \in \logbounded(\nu)$.
Then, for $\mu\times\nu$-almost every $(x,y)$, the limit
$$
\lim_{n \to \infty} \mumean_{(h(y), \dots, h(S^{n-1}y))} \big( g(x), g(Tx), \dots, g(T^{n-1} x) \big)
$$
exists and equals
\begin{equation}\label{e.func_Muirhead}
\mumean_h(g) \coloneqq \left[\smean\left(g^h\right)\right]^{1/\int h \dd \nu} \, .
\end{equation}
\end{corollary}

\begin{remark}
Formulas \eqref{e.HS} and \eqref{e.func_Muirhead} can be though as ``continuous'' (i.e., functional)
versions of the symmetric and the Muirhead means, respectively.
\end{remark}

\section{Open questions and directions for future research}\label{s.questions}


Permanents and Sinkhorn decompositions also make sense for \emph{multidimensional} matrices and functions (see e.g.\ \cite{bapat82}). We believe that most of the results of this paper can be extended accordingly, but we have not checked the details.

After having obtained a law of large numbers, a natural step is to look for a central limit theorem. In the particular case of symmetric means covered by \cite{HS} and Theorem~\ref{t.HS} above, a central limit theorem was obtained by Szek\'ely  \cite{Szekely_CLT}, under the assumption of independence.

Concerning the Law of Large Permanents itself, we do not believe that the statement of Theorem~\ref{t.LLP} is the optimal one.
We would like to weaken the hypothesis that  $f$ is (essentially) bounded away from $0$ and $\infty$.
This assumption was used twice in the proof: first, to apply Theorem~\ref{t.sink_for_bounded},
and second, to use the regularity estimates for the permanent from \S~\ref{ss.reg_per}.
So in order to strengthen the Law of Large Permanents we will probably need to solve other problems which are themselves of independent interest:
\begin{itemize}
	
\item To obtain explicit necessary and sufficient conditions on the function $f$
for the existence of a functional Sinkhorn decomposition $f = \phi g \psi$ in the sense of Definition~\ref{def.func_sink}, or at least in the more restricted setting of direct products (\S~\ref{ss.XY}).

\item To improve the estimates from \S~\ref{ss.reg_per} so that in particular we are able to deal with matrices containing zero entries.
\end{itemize}

The last and maybe most interesting line of research motivated by the results of this paper is
to extend the Law of Large Permanents to infinite matrices whose entries form a $\N^2$-indexed stochastic process (or, in more dynamical terms, whose entries are the values of a given function along the orbit of an ergodic $\N^2$-action).
Let us be more precise.

Suppose that $T$ is an ergodic measure-preserving action of the semigroup $\N^2$ on a Lebesgue probability space $(\Omega,\cA,\P)$.
Given a function $f:\Omega \to \R$ we define the $n \times n$ matrix $\dmat{f}{n} (\omega)\coloneqq f(T^{(i,j)}(\omega))_{0 \leq i,j \leq n-1}$.
Let $\cA_1$ and $\cA_2$ be the sub-$\sigma$-algebras formed by the $T^{(1,0)}$-invariant and the $T^{(0,1)}$-invariant sets, respectively.
We believe that the following statement should hold:

\begin{conjecture}  \label{conj.1}
If $\log f \in L^{\infty}(\P)$ then for $\P$-a.e.\ $\omega$,
\begin{equation*}
\lim_{n \to \infty} \pmean \left( \dmat{f}{n}(\omega) \right) = \smean_{\cA_1, \cA_2}(f).
\end{equation*}
\end{conjecture}

Conjecture~\ref{conj.1} would follow from the next statement regarding sequences of matrices:

\begin{conjecture}  \label{conj.2}
Let $(A_n)$ be a sequence of matrices of increasing sizes $n \times n$, all with row and column arithmetic means equal to $1$ (i.e.\ $\tfrac{1}{n} A_n$ is doubly stochastic for each $n$.)
Suppose that there exists $\lambda > 1$ each matrix $A_n$ is $\lambda$-bounded in the sense \eqref{e.lambda_bounded}.
Then
\begin{equation*}
\lim_{n \to \infty} \pmean(A_n) =1.
\end{equation*}
\end{conjecture}

It turns out that there are at least two claims in the literature at least as strong as Conjecture~\ref{conj.2}: see the references \cite{gi} and \cite{mc}.
Unfortunately we are unable to verify the correctness of either.

\bigskip
\begin{ack}
We thank M.~Courdurier and M.~Schraudner for helpful discussions
about concavity and $\N^2$-actions, respectively and to  V.L.~Girko for sending us a copy of \cite{gi}.
\end{ack}


\bigskip

\bigskip

\begin{small}
	\noindent
	\begin{tabular}{lll}
		\textsc{Jairo Bochi} &
		\textsc{Godofredo Iommi} &
		\textsc{Mario Ponce} \\
		\email{\href{mailto:jairo.bochi@mat.puc.cl}{jairo.bochi@mat.puc.cl}} &
		\email{\href{mailto:giommi@mat.puc.cl}{giommi@mat.puc.cl}} &		
		\email{\href{mailto:mponcea@mat.puc.cl}{mponcea@mat.puc.cl}} \\
		\href{http://www.mat.uc.cl/~jairo.bochi}{www.mat.uc.cl/$\sim$jairo.bochi} &
		\href{http://www.mat.uc.cl/~giommi}{www.mat.uc.cl/$\sim$giommi} & 		
		\href{http://www.mat.uc.cl/~mponcea}{www.mat.uc.cl/$\sim$mponcea} 	
	\end{tabular}
	
	\bigskip
	
	\noindent
	\textsc{Facultad de Matem\'aticas, Pontificia Universidad Cat\'olica de Chile}
	
	\noindent
	\textsc{Avenida Vicu\~na Mackenna 4860, Santiago, Chile}
\end{small}

\end{document}